\documentclass[11pt, a4paper]{article}  

\usepackage[hidelinks, pdftex]{hyperref}
\usepackage[numbers]{natbib}
\usepackage{algpseudocode}
\usepackage{amsthm, bm}

\usepackage{algorithm}
\usepackage{pifont}
\usepackage{amsmath}
\usepackage[nameinlink]{cleveref}
\usepackage{amssymb}

\usepackage{color}
\usepackage{csvsimple}
\usepackage{enumerate, enumitem}
\usepackage{mathrsfs} 
\usepackage{titlesec}
\usepackage{multicol, multirow, tabularx}
\usepackage{mathtools, calc}
\usepackage{verbatim}
\usepackage{listings}
\usepackage{graphicx}
\usepackage{lscape, textcomp}
\usepackage{tikz, subcaption}
\usepackage{xparse}

\usepackage{url}
\usepackage{enumitem}
\usepackage[tableposition=top]{caption}
\usepackage[toc,page]{appendix}
\usepackage[british]{babel}
\usepackage{tkz-euclide}
\usetikzlibrary{arrows}
\usepackage{fancyhdr} 
\usepackage{array, stackengine}
\usepackage{pbox}
\usepackage{sectsty}
\usepackage{tablefootnote, longtable}

\sectionfont{\fontsize{12}{15}\selectfont}
\subsectionfont{\fontsize{11}{15}\selectfont}
\subsubsectionfont{\fontsize{11}{15}\selectfont}

\newcolumntype{L}[1]{>{\raggedright\let\newline\\\arraybackslash\hspace{0pt}}m{#1}}
\newcolumntype{C}[1]{>{\centering\let\newline\\\arraybackslash\hspace{0pt}}m{#1}}
\newcolumntype{R}[1]{>{\raggedleft\let\newline\\\arraybackslash\hspace{0pt}}m{#1}}

\newcommand{\mtx}[1]{\boldsymbol{#1}}
\newcommand{\mvec}[1]{\boldsymbol{#1}}

\newcommand\Item[1][]{%
	\ifx\relax#1\relax  \item \else \item[#1] \fi
	\abovedisplayskip=0pt\abovedisplayshortskip=0pt~\vspace*{-\baselineskip}}

\makeatletter
\newcounter{algorithmicH}
\let\oldalgorithmic\algorithmic
\renewcommand{\algorithmic}{%
	\stepcounter{algorithmicH}
	\oldalgorithmic}
\renewcommand{\theHALG@line}{ALG@line.\thealgorithmicH.\arabic{ALG@line}}
\makeatother


\makeatletter
\def\th@plain{%
	\thm@notefont{}
	\itshape 
}
\def\th@definition{%
	\thm@notefont{}
	\normalfont 
}
\makeatother

\theoremstyle{plain}
\newtheorem{theorem}{Theorem}[section]
\newtheorem{lemma}[theorem]{Lemma}

\newtheorem{corollary}[theorem]{Corollary}

\newtheorem{assump}[theorem]{Assumption}

\theoremstyle{definition}
\newtheorem{example}[theorem]{Example}
\newtheorem{definition}[theorem]{Definition}
\newtheorem*{rem*}{Remark}
\newtheorem*{warning*}{Warning}
\newtheorem{rem}[theorem]{Remark}

\DeclareMathOperator{\rank}{rank}

\DeclareMathOperator{\range}{range}
\DeclareMathOperator{\prob}{\mathbb{P}}

\DeclareMathOperator*{\argmin}{\arg\!\min}
\DeclareMathOperator{\vc}{vec}

\NewDocumentCommand{\ceil}{s O{} m}{%
	\IfBooleanTF{#1} 
	{\left\lceil#3\right\rceil} 
	{#2\lceil#3#2\rceil} 
}

\newsavebox\CBox

\makeatletter
\newcommand{\subalign}[2][c]{%
	\if#1c\vcenter\else\vtop\fi{%
		\Let@ \restore@math@cr \default@tag
		\baselineskip\fontdimen10 \scriptfont\tw@
		\advance\baselineskip\fontdimen12 \scriptfont\tw@
		\lineskip\thr@@\fontdimen8 \scriptfont\thr@@
		\lineskiplimit\lineskip
		\ialign{\hfil$\m@th\scriptstyle##$&$\m@th\scriptstyle{}##$\hfil\crcr
			#2\crcr
		}%
	}%
}
\makeatother

\makeatletter
\renewcommand*{\@fnsymbol}[1]{\ensuremath{\ifcase#1\or *\or \ddagger\or \mathsection\or \vee\or \wedge\or \dagger\or
		\mathsection\or \mathparagraph\or \|\or **\or \dagger\dagger
		\or \ddagger\ddagger \else\@ctrerr\fi}}
	
\numberwithin{equation}{section}
\begin{document}
\title{A dimensionality reduction technique for unconstrained global optimization of functions with low effective dimensionality\thanks{This work was supported by The Alan Turing Institute under The Engineering and Physical Sciences Research Council (EPSRC) grant EP/N510129/1}}

\author{
	Coralia Cartis \thanks{The Alan Turing Institute, The British Library, London, NW1 2DB, UK} \textsuperscript{\normalfont,}\thanks{Mathematical Institute, University of Oxford, Radcliffe Observatory Quarter, Woodstock Road,
		Oxford, OX2 6GG, UK; \texttt{cartis,otemissov@maths.ox.ac.uk}}
	\and
	Adilet Otemissov \footnotemark[2] \textsuperscript{\normalfont,}\footnotemark[3] 
}

\date{16th January 2020}
\maketitle
\footnotesep=0.4cm
{\small
	\begin{abstract}
		We investigate the unconstrained global optimization of functions with low effective dimensionality, that are constant along certain (unknown) linear subspaces. Extending the technique of random subspace embeddings in [Wang et al., Bayesian optimization in a billion dimensions via random embeddings. \textit{JAIR}, 55(1): 361--387, 2016], we study a generic Random Embeddings for Global Optimization (REGO) framework that is compatible with any global minimization algorithm. Instead of the original, potentially large-scale optimization problem, within REGO, a Gaussian random, low-dimensional problem with bound constraints is formulated and solved in a reduced space. We provide  novel probabilistic bounds for the success of REGO in solving the  original, low effective-dimensionality problem, which show its independence of the (potentially large) ambient dimension and its precise dependence on the dimensions of the effective and randomly embedding subspaces. These results significantly improve existing theoretical analyses by providing the exact distribution of a reduced minimizer and its Euclidean norm and by the general assumptions required on the problem. We validate our theoretical findings by extensive numerical testing of REGO with three types of global optimization solvers, illustrating the improved scalability of REGO compared to the full-dimensional application of the respective solvers.
	\end{abstract}
	
	\bigskip
	
	\begin{center}
		\textbf{Keywords:}
		global optimization, random matrix theory, dimensionality reduction techniques, functions with low effective dimensionality
	\end{center}
}

\maketitle               

\section{Introduction}
In this paper, we address the unconstrained global optimization problem 
\begin{equation}
\tag{P}
\begin{aligned} \label{eq: GO}
\min_{\mvec{x}\in \mathbb{R}^D} & \;\; f(\mvec{x}), \\
\end{aligned}
\end{equation}
where $f: \mathbb{R}^D \rightarrow \mathbb{R}$ is a real-valued continuous, possibly non-convex, deterministic, function defined on the whole $\mathbb{R}^D$. We assume that there exists $\mvec{x}^* \in \mathbb{R}^D$ such that $\min_{\mvec{x} \in \mathbb{R}^D} f(\mvec{x}) = f(\mvec{x}^*) = f^*$. This implies that $f$ is bounded below, namely, $f^* > - \infty$, and that the minimum in \eqref{eq: GO} is attained (not all minimizers are at infinity).

To alleviate the curse of dimensionality, we further restrict ourselves to a particular class of functions whose true (intrinsic) dimension is much less than the ambient problem dimension. These functions are constant along certain linear subspaces, which may not necessarily be aligned with the standard axes. In literature, these functions are known under different names: functions with `\textit{low effective dimensionality}' \cite{Wang2016}, functions with `\textit{active subspaces}' \cite{Constantine2015} and `\textit{multi-ridge}' functions \cite{Fornasier2012, Tyagi2014}.
They have been found in a number of applications mainly related to parameter studies. In hyper-parameter optimization for neural networks \cite{Bergstra2012} and heuristic algorithms for combinatorial optimization problems \cite{Hutter2014}, studies have shown that the respective objective functions are affected by only a few hyper-parameters while the many other input hyper-parameters are redundant. Similarly, in complex engineering and physical simulation problems \cite{Constantine2015}, such as in climate modelling \cite{Knight2007}, systems are modelled by several input parameters with only a small number of the parameters or a combination of them having a true effect on the system's behaviour.
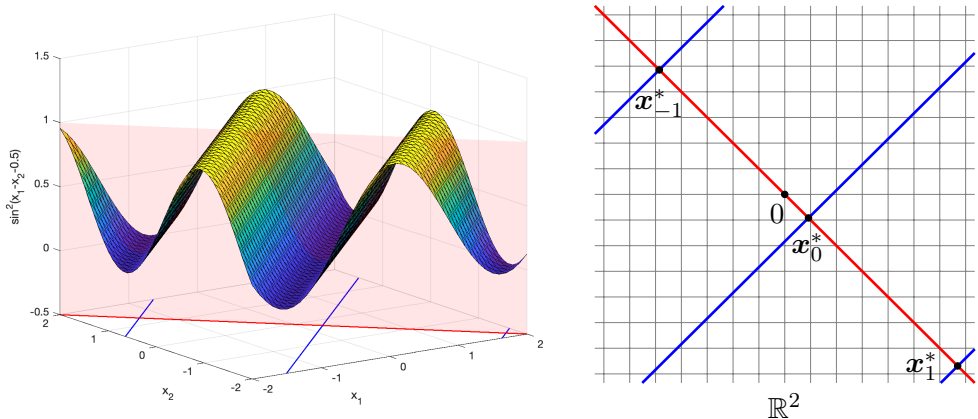
\begin{figure}[!t]
	\centering
	\includegraphics[scale = 0.4]{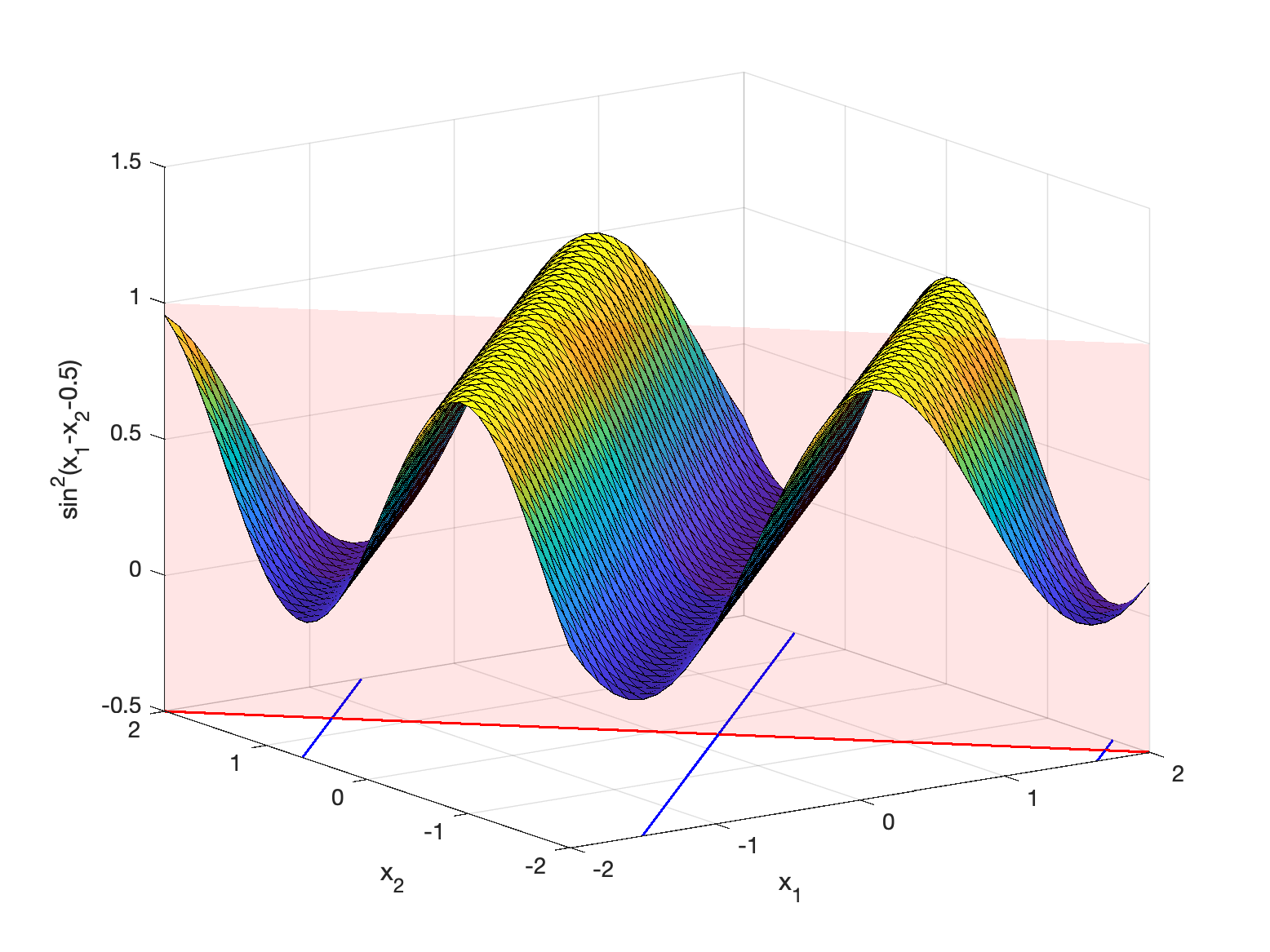}
	\begin{tikzpicture}
	\node at (0,-2.8) {$\mathbb{R}^2$};
	\draw[step=0.34cm,gray,very thin,] (-2.5,-2.5) grid (2.5,2.5);
	\draw[line width = 1, color = blue] (-1.875, -2.5) -- (2.5,1.875);
	\draw[line width = 1, color = blue] (2.052, -2.5) -- (2.5,-2.052);
	\draw[line width = 1, color = blue] (-2.5, 0.802) -- (-0.802,2.5);
	
	\draw[line width = 1, color = red] (-2.5, 2.5) -- (2.5,-2.5);
	\filldraw (0,0) circle (1.2pt);
	\filldraw (0.3125,-0.3125) circle (1.2pt);
	\node at (-0.1, -0.25) {$0$};
	
	\filldraw (91/40,-91/40) circle (1.2pt);
	\filldraw (-33/20,+33/20) circle (1.2pt);
	\node at (-33/20, 1.2) {$\mvec{x}_{-1}^*$};
	\node at (0.3, -0.7) {$\mvec{x}_0^*$};
	\node at (1.8, -91/40) {$\mvec{x}_1^*$};
	\end{tikzpicture}
	\caption{The function in \eqref{eq: example} and its domain are plotted on the left-  and right-side, respectively. The red line  is the effective subspace $\mvec{x} = (1 \; -1)y$ and it intersects the blue lines of global minimizers at  $\mvec{x}_k^*$  defined in \Cref{ex: sin_function} for $k = -1,0,1$; these points also correspond to optimal solutions in the reduced space. }
	\label{fig: example_sin_function}		
\end{figure}
To clarify this concept, we give a simple example of a function with lower effective dimensionality.
\begin{example}\label{ex: sin_function}
	Consider the following optimization problem:
	\begin{equation} \label{eq: example}
	\begin{aligned}
	\min_{\mvec{x} \in \mathbb{R}^2} & \;\; f(\mvec{x}) = \sin^2(x_1-x_2-0.5). \\
	\end{aligned}
	\end{equation} 
	By solving $   f(\mvec{x}) = 0$, we find that the set of global minimizers is given by $\{ (1\;\; 1)^T t - (0 \;\; 0.5+\pi k)^T: t \in \mathbb{R}, k \in \mathbb{Z} \}$. For each fixed value of $k$, the set corresponds to a distinct line of global minimizers along which the function is constant (see \Cref{fig: example_sin_function}). The effective subspace\footnote{The effective subspace can be determined by considering the orthogonal complement of the constant subspace (along which $f$ does not vary), in this example, spanned by the vector $(1 \; 1)^T$.} of $f$ is $(x_1 \; x_2) = (1 \; -1)^T y$ for $y \in \mathbb{R}$. We substitute this in \eqref{eq: example} to obtain the reduced/lower-dimensional optimization problem  $\min_{y \in \mathbb{R}} \sin^2(2y-0.5)$, which has the same global minimum as \eqref{eq: example}, with global minimizers $y_k^* = \pi k/2 + 0.25$, $k \in \mathbb{Z}$. We recover the corresponding solutions to \eqref{eq: example} by setting $\mvec{x}_k^* = (1\;-1)^T y_k^*$ for $k \in \mathbb{Z}$.
\end{example}

As \Cref{ex: sin_function} illustrates, it is possible to cast \eqref{eq: GO} into a lower-dimensional problem which has the same global minimum $f^*$. This is straightforward when the effective subspace is known, but far less so in applications where $f$ is potentially black-box. When the effective subspace is unknown, it was proposed (in the context of Bayesian Optimization) in Wang et al. \cite{Wang2016} to use random embeddings. The proposed technique solves the following lower-dimensional optimization problem instead of directly tackling \eqref{eq: GO}:
\begin{equation} \label{eq: REGO}
\tag{RP}
\begin{aligned}
\min_{\mvec{y}}  &\;\; f(\mtx{A}\mvec{y})  \\
\text{subject to} &\;\; \mvec{y} \in \mathcal{Y} = [-\delta, \delta]^d, 
\end{aligned}
\end{equation}
where $\mtx{A}$ is a $D\times d$ Gaussian random matrix (see \Cref{def: Gaussian_matrix}) and $\mathcal{Y} = [-\delta, \delta]^d$ for some carefully chosen $\delta > 0$ and $d \ll D$. Note that, unlike \eqref{eq: GO}, \eqref{eq: REGO} has (box) constraints, which are typically imposed to make the approach practical (i.e. to avoid unrealistic searches over infinite domains).
\begin{definition} \label{def: successful_REGO}
	We say that \eqref{eq: REGO} is \textit{successful} if there exists $\mvec{y}^* \in \mathcal{Y}$ such that $f(\mtx{A}\mvec{y}^*) = f^*$.
\end{definition}

\paragraph{Related work.} The scalability challenges of Bayesian Optimization (BO) algorithms for generic black-box functions have  prompted research into improving efficiency of this class of methods for functions with special structure. 
 Different structural assumptions on the objective have been analysed for BO, such as  additivity or  (partial) separability, which assumes that the objective function can be represented as the sum of smaller-dimensional functions with non-overlapping variables \cite{Wang2018, Kandasamy2015, Li2016} or with overlapping ones \cite{Rolland2018}. 
 
 Another popular structural assumption is the above-mentioned low-effective dimensionality of the objective. In its simplest form, this considers the effective subspace to be aligned with the coordinate axes, which is equivalent to the presence of redundant variables \cite{Chen2012, BenSalem2019}. 
 More generally, the optimization of functions that are constant along {\it arbitrary} linear subspaces -- which, as mentioned above, is also the focus of this paper -- has been addressed using BO methods in  \cite{Djolonga2013, Wang2016, Garnett2014, Eriksson2018}, and extended to other problem and algorithm classes such as derivative-free optimization \cite{QianHuYu2016}, multi-objective optimization \cite{Qian2017} and evolutionary methods  \cite{Sanyang2016}. 
 Some proposals learn the effective subspace of the function beforehand (using for example, a low rank matrix recovery approach) \cite{Tyagi2014, Fornasier2012} and then optimize in the reduced subspace \cite{Djolonga2013, Eriksson2018}.  Alternating learning and optimization steps has also been proposed \cite{Garnett2014}, as well as bypassing learning and directly optimizing in randomly-chosen low-dimensional subspaces (provided an estimate of the effective dimension is known) \cite{Wang2016, Binois2014, Binois2017}. 
 
 For the latter, \citeauthor{Wang2016} \cite{Wang2016} developed the so-called REMBO algorithm, which is a BO framework for problem \eqref{eq: GO} with box constraints $x\in \mathcal{X}$ that uses  Gaussian random embeddings (namely, $\mtx{A}$ is a Gaussian random matrix)  to generate the reduced problem \eqref{eq: REGO}. They find that the size of $\mathcal{Y}$ is the primary factor in determining the success (or failure) of the reduced problem, and quantify the probability of success of \eqref{eq: REGO} for the case when the embedded dimension $d$ is equal to the effective one and the effective subspace is aligned with the coordinate axes (see \cite[Theorem 3]{Wang2016}).  A challenge of \eqref{eq: REGO} for BO with box constraints is that, even when \eqref{eq: REGO} is successful,  the high-dimensional image $\mtx{A}\mvec{y} \in \mathbb{R}^D$ of a point $\mvec{y} \in \mathcal{Y}$ may be outside the feasible set $\mathcal{X}$. For this reason, REMBO is equipped with a map $p_{\mathcal{X}}:\mathbb{R}^D \rightarrow \mathbb{R}^D$ that projects the image of the reduced solutions that fall outside $\mathcal{X}$ to the closest point on the boundary of $\mathcal{X}$. To model a Gaussian Process for the reduced problem,  \cite{Wang2016} proposes two kernels: a high-dimensional $k_{\mathcal{X}}$ and a low-dimensional $k_{\mathcal{Y}}$. Kernel $k_{\mathcal{X}}$ suffers from high-dimensionality as it constructs a GP in a $D$-dimensional space. The benefit of $k_{\mathcal{Y}}$ is that it constructs a GP in a $d$-dimensional subspace,  but this kernel over-explores regions in $\mathcal{Y}$ whose high-dimensional images outside $\mathcal{X}$ are mapped to the same points in $\mathcal{X}$ 
 through a non-injective $p_{\mathcal{X}}$.
 To remedy these issues, \citeauthor{Binois2014} \cite{Binois2014} propose a new kernel $k_{\Psi}$ which has the benefit of being low-dimensional while avoiding the over-exploratory tendency of $k_{\mathcal{Y}}$. In \cite{Binois2017}, \citeauthor{Binois2017} also propose a new mapping $\gamma$ (instead of $p_{\mathcal{X}}$) and define $\mathcal{Y}$ and new kernels based on this new mapping. 
 
  \citeauthor{Sanyang2016} \cite{Sanyang2016} develop REMEDA, which uses random embeddings  within an Evolutionary algorithm EDA. Their theoretical results on quantifying the size of $\mathcal{Y}$/the success of \eqref{eq: REGO} improve on those in \cite{Wang2016} and are 
 applicable for certain choices of $d$, greater than the effective subspace dimension; they also experiment with estimating the effective dimension numerically.  
 
  \citeauthor{QianHuYu2016} \cite{QianHuYu2016} extend the framework and some of the results in \citeauthor{Wang2016} \cite{Wang2016} to functions with approximate low effective subspaces, proposing the use of multiple random embeddings  within any derivative-free solver. They contrast the use of a single versus multiple embeddings on three test problems of varying dimensions and using three different types of derivative-free solvers (evolutionary, Bayesian and model-based). 
 
 Recently, in the context of Bayesian optimization, \citeauthor{Nayebi2019} \cite{Nayebi2019} use a different random ensemble based on hashing matrices to represent the embedded subspaces and define $\mathcal{Y}$  as $[-1,1]^d$; this formulation guarantees that the high-dimensional points are always inside $\mathcal{X}$ and, thus, their method avoids the feasibility corrections of REMBO.
 
 \paragraph{Our contributions.}
 We investigate a general random embeddings framework for unconstrained global optimization of functions with low effective dimensionality, where we allow the effective subspace of the objective function and its dimension (denoted by $d_e$) to be arbitrary (not necessarily aligned with coordinate axes and not limited in dimension by problem constants)\footnote{Note that, as problem \eqref{eq: GO} has no (bound) constraints, we do not need to use the projection operator $p_{\mathcal{X}}$ in \cite{Wang2016, Binois2014, Binois2017}. Of course,  this comes at the cost of our approach being unable to guarantee feasibility for the original problem if \eqref{eq: GO} does have constraints.}. This framework also allows the use of any global solver to solve the reduced problem. 
 
 We significantly extend and improve the theoretical analyses in \cite{Wang2016, Sanyang2016}, providing an in-depth investigation of the reduced problem \eqref{eq: REGO} when Gaussian random embeddings are used. In particular, while \cite{Wang2016, Sanyang2016} estimate the Euclidean norm of a (random) reduced minimizer, we derive its exact distribution, 
 using tools from random matrix theory.
 We show that this reduced minimizer, when appropriately scaled, follows the inverse chi-squared distribution  with $d-d_e+1$ degrees of freedom, where $d$ is the dimension of the random embedding (\Cref{thm: x^*_T/y^*_2_follow_chi_square}). Moreover, we derive the probability density function of this reduced minimizer (\Cref{thm: pdf_of_y^*}) by first proving that it follows a spherical distribution. These results imply that, under  certain assumptions, solving \eqref{eq: REGO} has no dependence on the ambient dimension $D$. Subsequently, \Cref{thm: prob_REGO_is_successful} and \Cref{cor: REGO_successful_lower_bound}  estimate the probability that \eqref{eq: REGO} is successful.  The latter result 
 extends  both \cite[Theorem 3]{Wang2016} and \cite[Theorem 2]{Sanyang2016} to arbitrary effective subspaces and any $d \geq d_e$, and establishes a notable and more precise  trade-off between the success of \eqref{eq: REGO}, $\delta$ (the size of the reduced domain $\mathcal{Y}$) and the embedding  dimension $d$; thus allowing us to choose appropriate values for these parameters in the algorithm. 
 Furthermore, we describe how to extend the main results to affine random embeddings (which draw random subspaces at any chosen (reference) point in $\mathbb{R}^D$), which indicate that the probability of success of \eqref{eq: REGO} is higher if the point of reference is closer to the set of global minimizers.
 	
Similarly to the algorithmic frameworks proposed in \cite{Wang2016, QianHuYu2016}, we propose REGO (Random Embeddings for Global Optimization) that solves a single randomly-embedded reduced problem \eqref{eq: REGO} instead of \eqref{eq: GO} and is compatible with any generic  global optimization solver\footnote{The REGO framework is defined for the unconstrained \eqref{eq: GO} but can also be helpful for constrained problems (for example, $\mvec{x} \in \mathcal{X}$), where the constraints are imposed just to avoid searches over an infinite domain and where minimizers outside the feasible domain are acceptable.}. We use and validate our theoretical results by
providing extensive numerical testing of REGO with three types of solvers for \eqref{eq: REGO}: DIRECT (Lipschitz-optimization), BARON (branch and bound), and KNITRO (multi-start local optimization). We use $19$ standard global optimization test problems to generate functions with effective dimensionality structure and of growing ambient dimension $D$. When comparing REGO with the direct optimization of the ensuing problems without embeddings, we find  that REGO's performance is essentially independent of $D$ for all three solvers and that it is successful in recovering the original global minimum in most cases with only one embedding\footnote{These numerical results assume that (an upper bound on) the true effective dimension $d_e$ is known/available.}. We also test the robustness of REGO's performance to variations in algorithm parameters such as  $\delta$ and $d$.

\paragraph{Paper outline.}
In \Cref{sec: functions with low eff. dim.}, we formally define and describe functions with low effective dimensionality emphasizing their geometrical aspects. In \Cref{sec: characterization}, we characterize the reduced minimizers in the reduced space focusing on the minimal 2-norm minimizer. For this minimizer, we derive the distribution of its Euclidean norm and its probability density function. We use the former result in \Cref{sec: bounding_success} to derive a probabilistic bound for the success of \eqref{eq: REGO}. In \Cref{sec: Numerical_experiments}, we conduct numerical experiments to test REGO algorithm on functions with low effective dimensionality using three optimization solvers, namely, DIRECT, BARON and KNITRO, while in \Cref{sec: conclusions} we draw our conclusions and future directions.

\paragraph{Notation.}
We use bold capital letters to denote matrices ($\mtx{A}$) and bold lowercase letters ($\mvec{a}$) to denote vectors. In particular, we use $\mtx{I}_D$ to denote the $D \times D$ identity matrix and $\mvec{0}_D$, $\mvec{1}_D$ (or simply $\mvec{0}$, $\mvec{1}$) to denote the $D$-dimensional vectors of zeros and ones, respectively. For an $D \times d$ matrix $\mtx{A}$, we write $\range(\mtx{A})$ to denote the linear subspace spanned by the columns of $\mtx{A}$ in $\mathbb{R}^D$. 

We let $\langle \cdot , \cdot \rangle$, $\| \cdot \|$ and $\| \cdot \|_{\infty}$ to denote the usual Euclidean inner product, the Euclidean norm and the infinity norm, respectively. Where emphasis is needed, for the Euclidean norm we also use $\| \cdot \|_2$. 

Given two random variables (vectors) $x$ and $y$ ($\mvec{x}$ and $\mvec{y}$), we write $x \stackrel{law}{=} y$ ($\mvec{x} \stackrel{law}{=} \mvec{y}$) to denote the fact that $x$ and $y$ ($\mvec{x}$ and $\mvec{y}$) have the same distribution. We reserve the letter $\mtx{A}$ to refer to a $D\times d$ Gaussian random matrix (see \Cref{def: Gaussian_matrix}) and write $\chi^2_n$ to denote a chi-squared random variable with $n$ degrees of freedom (see \Cref{def: chi-squared_rv}).

\section{Functions with low effective dimensionality} \label{sec: functions with low eff. dim.}
In this section, we formally define functions with low effective dimensionality and describe the geometry of \eqref{eq: REGO}. 

\subsection{Definitions and assumptions}
Functions with low effective dimensionality can be defined in at least two ways \cite{Fornasier2012,Wang2016}. We will work with a definition given in terms of linear subspaces, provided in \cite{Wang2016}.
\begin{definition}[Functions with low effective dimensionality] \label{def: simple_function}
	A function $f : \mathbb{R}^D \rightarrow \mathbb{R}$ has effective dimensionality $d_e \leq D$ if there exists a linear subspace $\mathcal{T}$ of dimension $d_e$ such that for all vectors $\mvec{x}_{\top}$ in $\mathcal{T}$ and $\mvec{x}_{\perp}$ in $\mathcal{T}^{\perp}$ (orthogonal complement of $\mathcal{T}$) we have
	\begin{equation}\label{eq: def_fun_eff_dim}
		f(\mvec{x}_{\top} + \mvec{x}_{\perp}) = f(\mvec{x}_{\top}).
	\end{equation}
	and $d_e$ is the smallest integer satisfying \eqref{eq: def_fun_eff_dim}.
\end{definition}
The linear subspace $\mathcal{T}$ is called the \textit{effective} subspace of $f$ and its orthogonal complement $\mathcal{T}^{\perp}$, the \textit{constant} subspace of $f$. It is convenient to think of $\mathcal{T}^{\perp}$ as a subspace of no variation of largest dimension (along which the value of $f$ does not change) and $\mathcal{T}$ as its orthogonal complement. 

Every vector $\mvec{x}$ can be decomposed as $\mvec{x} = \mvec{x}_{\top} + \mvec{x}_{\perp}$, where $\mvec{x}_{\top}$ and $\mvec{x}_{\perp}$ are orthogonal projections of $\mvec{x}$ onto $\mathcal{T}$ and $\mathcal{T}^{\perp}$, respectively. In particular, if $\mvec{x}^*$ is a global minimizer and $f^*$ is the global minimum of $f$ in $\mathcal{X}$ then $\mvec{x}^* = \mvec{x}^*_{\top} + \mvec{x}^*_{\perp}$ and 
\begin{equation}\label{eq: f^*=f(x_top^*)}
f^* = f(\mvec{x}^*) = f(\mvec{x}^*_{\top} + \mvec{x}^*_{\perp}) = f(\mvec{x}^*_{\top}).
\end{equation}
Moreover, we have 
\begin{equation*}
f^* = f(\mvec{x}_{\top}^*) = f(\mvec{x}_{\top}^* + \mvec{x}_{\perp}) 
\end{equation*}
for every vector $\mvec{x}_{\perp}$ in $\mathcal{T}^{\perp}$. It is important to note that there can be multiple points $\mvec{x}_{\top}^*$ in $\mathbb{R}^D$ satisfying the above definition such as, for instance, $\mvec{x}^*_{-1}$, $\mvec{x}^*_{0}$ and $\mvec{x}^*_{1}$ in \Cref{ex: sin_function}. By contrast, the function $f = (x_1 - x_2-0.5)^2$ admits a unique $\mvec{x}_{\top}^*$ given by $(0.25 \; -0.25)^T$. 

We summarize the above discussion in the following assumption.
\begin{assump}\label{ass: REGO_fun_eff_dim}
	The function $f:\mathbb{R}^D \rightarrow \mathbb{R}$ is continuous and has effective dimensionality $d_e \leq d$ with effective subspace\footnote{Note that $\mathcal{T}$ in \Cref{ass: REGO_fun_eff_dim} may not be aligned with the standard axes. } $\mathcal{T}$ and constant subspace $\mathcal{T}^{\perp}$ spanned by the columns of orthonormal matrices $\mtx{U} \in \mathbb{R}^{D \times d_e}$ and $\mtx{V} \in \mathbb{R}^{D\times(D-d_e)}$, respectively.
\end{assump}
Recalling the definition of problem \eqref{eq: GO} on page \pageref{eq: GO}, let
$$ \mathcal{G} = \{\mvec{x} \in \mathbb{R}^D: f(\mvec{x}) = f^* \} $$
be the set of global minimizers in $\mathbb{R}^D$. Under \Cref{ass: REGO_fun_eff_dim}, the set $\mathcal{G}$ can be represented as a union of (possibly infinitely many) affine subspaces each containing one particular $\mvec{x}_{\top}^*$ (see proof of \Cref{thm: prob_REGO_is_successful}). 
Each of these affine subspaces is a $(D-d_e)$-dimensional set $\{ \mvec{x} \in \mathbb{R}^D : \mvec{x} \in \mvec{x}_{\top}^* + \mathcal{T}^{\perp} \}$ --- a translation of $\mathcal{T}^{\perp}$ by the vector $\mvec{x}_{\top}^*$ that the corresponding affine subspace must contain. In particular, if there is a unique $\mvec{x}_{\top}^*$ in $\mathbb{R}^D$ then $\mathcal{G} = \{ \mvec{x} \in \mathbb{R}^D : \mvec{x} \in \mvec{x}_{\top}^* + \mathcal{T}^{\perp} \}$. Note also that point(s) $\mvec{x}^*_{\top}$ lie in $\mathcal{G} \cap \mathcal{T}$, and are the closest minimizers to the origin in Euclidean norm amongst all the minimizers lying in their respective affine subspaces.

Our analysis applies to any minimizer $\mvec{x}^*$ with $\mvec{x}_{\top}^* \neq \mvec{0}$. If $\mvec{x}_{\top}^* = \mvec{0}$, then \eqref{eq: REGO} has a trivial solution. In that case, the origin is a global minimizer implying that every embedding is successful with a solution $\mvec{y}^* = \mvec{0}$. Hence, we focus our analysis for finding a(ny) minimizer $\mvec{x}^* \in \mathcal{G}$ with $\mvec{x}_{\top}^* \neq \mvec{0}$. 
\begin{assump} \label{def: REGO_choice_of_x^*}
	Given \Cref{ass: REGO_fun_eff_dim}, let $\mvec{x}^* \in \mathcal{G}$ such that $\mvec{x}_{\top}^* = \mtx{U}\mtx{U}^T\mvec{x}^*$  --- the unique Euclidean projection of $\mvec{x}^*$ onto $\mathcal{T}$ --- is non-zero.  Let $\mathcal{G}^* := \mvec{x}_{\top}^* + \mathcal{T}^{\perp}$ be the affine subspace of $\mathcal{G}$ that contains $\mvec{x}_{\top}^*$. 
\end{assump}

The set $\mathcal{G}$ contains infinitely many global minimizers --- a particularly useful feature of the functions with low effective dimensionality; this fact suggests that targeting $\mathcal{G}$ numerically may potentially be easier than if $\mathcal{G}$ contained only one point.

\subsection{Geometric description}
We now provide a geometric description of \eqref{eq: REGO}, which serves as a basis for our theoretical investigations.

In \Cref{fig: mapping_Y_to_X}, we illustrate schematically $\mathcal{T}$ (the effective subspace of $f$), $\mathcal{T}^{\perp}$ (the orthogonal component of $\mathcal{T}$), $\mathcal{G}^*$ (a connected component\footnote{Recall that $\mathcal{G}$ is a union of affine subspaces; $\mathcal{G}^*$ is one of them.} of $\mathcal{G}$), $\mvec{x}_{\top}^*$ (the orthogonal projection of the global minimizers on $\mathcal{G}^*$ onto $\mathcal{T}$). Since the orientation and position of these geometric objects are solely determined by the (deterministic) objective function, they are fixed, non-random. 

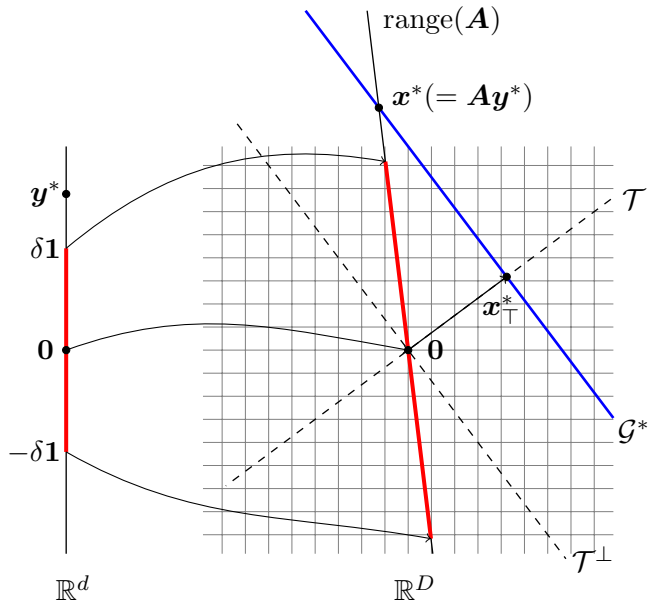
\begin{figure}[!t]
	\centering
	\begin{tikzpicture}[scale = 0.9]
	\draw[step=0.34cm,gray,very thin] (-3,-3) grid (3,3);
	\draw[line width = 1pt, color = blue] (-1.5,5) -- (3,-1);
	\draw[line width = 0.5pt, ->] (0,0) -- (36/25,27/25);
	\draw[line width = 0.5pt] (-0.6,5) -- (0.36,-3);
	\draw[line width = 1.5pt, color = red] (-1/3,25/9) -- (1/3,-25/9);
	\draw[line width = 0.5pt, dashed] (80/27, 20/9) -- (-72/27,-2);
	\draw[line width = 0.5pt, dashed] (-90/36, 90/27) -- (83/36,-83/27);
	
	\draw[line width = 0.5pt] (-4-1,-3) -- (-4-1,3);
	\draw[line width = 1.5pt, color = red] (-4-1,-1.5) -- (-4-1,1.5);
	\draw[->] (-4-1,1.5) to[out=40,in=170] (-1/3,25/9);
	\draw[->] (-4-1,-1.5) to[out=-30,in=170] (1/3,-25/9);
	\draw[->] (-5,0) to[out=20,in=170] (0,0);

	
	\draw[fill] (0,0) circle(1.5pt);
	\draw[fill] (36/25,27/25) circle(1.5pt);
	\draw[fill] (-3/7, 25/7) circle(1.5pt);
	\draw[fill] (-4-1, 2.3) circle(1.5pt);
	\draw[fill] (-5, 0) circle(1.5pt);
	
	\node at (3.3,20/9) {$\mathcal{T}$};
	\node at (2.7,-83/27) {$\mathcal{T}^{\perp}$};
	\node at (0.5,34/7) {$\range(\mtx{A})$};
	\node at (34/25,0.6) {$\mvec{x}_{\top}^*$};
	\node at (0.8,26/7) {$ \mvec{x}^* (=  \mtx{A}\mvec{y}^*)$};
	\node at (0.1,-3.5) {$\mathbb{R}^D$};
	\node at (-4.9,-3.5) {$\mathbb{R}^d$};
	\node[align = left] at (-4.3-1,1.5) {$\delta \mvec{1}$};
	\node[align = left] at (-4.47-1,-1.5) {$-\delta \mvec{1}$};
	\node at (-5.3,2.3) {$\mvec{y}^*$};
	\node at (-5.3,0) {$\mvec{0}$};
	\node at (3.3, -1.2) {$\mathcal{G}^*$};
	\node at (0.4,0) {$\mvec{0}$};

	\end{tikzpicture}
	\caption{The figure shows an abstract illustration of the embedding of a $d$-dimensional linear subspace into $\mathbb{R}^D$. The line $\range(\mtx{A})$ corresponds to the embedded subspace. The red line in $\mathbb{R}^d$ represents the hypercube $\mathcal{Y} = \{ \mvec{y} \in \mathbb{R}^d: -\delta \mvec{1} \leq \mvec{y} \leq \delta \mvec{1} \}$, which, after application of $\mtx{A}$, maps to the red line along $\range(\mtx{A})$ in $\mathbb{R}^D$. In this configuration, condition \eqref{cond: existence_of_y_in_R^d} is satisfied but \eqref{cond: existence_of_y_in_Y} is not: $\range(\mtx{A})$ intersects $\mathcal{G}$ at $\mvec{x}^* = \mtx{A}\mvec{y}^*$, but $\mvec{y}^*$ lies outside $\mathcal{Y}$. }
	\label{fig: mapping_Y_to_X}
\end{figure}

By applying the `random embedding' \eqref{eq: REGO}, we switch from optimizing over $\mathbb{R}^D$ to optimizing over $\mathcal{Y}$. The linear mapping $ \mvec{y} \rightarrow \mtx{A}\mvec{y}$ maps points of the hypercube $\mathcal{Y}$ to points along the subspace $\range(\mtx{A})$ in $\mathbb{R}^D$, which means that searching over $\mathcal{Y}$ is equivalent to searching over the corresponding feasible set along $\range(\mtx{A})$ in $\mathbb{R}^D$. An example of this mapping is illustrated in \Cref{fig: mapping_Y_to_X} with two red line segments: the segment (from $-\delta \mvec{1}$ to $\delta \mvec{1}$) representing $\mathcal{Y}$ is being mapped to the right segment, which lies in $\range(\mtx{A})$. It is important to note that the centre of $\mathcal{Y}$ maps to the origin in $\mathbb{R}^D$ and, hence, the corresponding search in the original space is also centred at the origin. 

The most valuable information that we want to retain while performing dimensionality reduction is the value of $f^*$. We would like $\min_{\mvec{y}\in \mathcal{Y}} f(\mtx{A}\mvec{y}) = f^*$, which holds only if there is at least one $\mvec{y}^*$ in $\mathcal{Y}$ such that $f(\mtx{A}\mvec{y}^*) = f^*$. This condition has a geometric interpretation and, following from the definition of $\mathcal{G}$, it can equivalently be stated as:
\begin{equation}\label{cond: existence_of_y_in_Y}
	\text{there exists a $\mvec{y}^* \in \mathcal{Y}$ such that $\mtx{A}\mvec{y}^* \in \mathcal{G}$.}
\end{equation}
For \eqref{cond: existence_of_y_in_Y} to hold, we must first ensure that 
\begin{equation}\label{cond: existence_of_y_in_R^d}
	\text{there exists a $\mvec{y}^*$ in $\mathbb{R}^d$ such that $\mtx{A}\mvec{y}^* \in \mathcal{G}$.}
\end{equation}
In this regard, Wang et al. \cite{Wang2016} proved the following theorem. 
\begin{theorem}\label{thm: Wang_theorem} (see \cite[Theorem 2]{Wang2016})
	Let \Cref{ass: REGO_fun_eff_dim} hold and let $\mtx{A}$ be a $D \times d$ Gaussian matrix with $d \geq d_e$. Then, with probability one, for any fixed $\mvec{x} \in \mathbb{R}^D$, there exists a $\mvec{y} \in \mathbb{R}^d$ such that $f(\mvec{x}) = f(\mtx{A}\mvec{y})$. In particular, for a global minimizer $\mvec{x}^*$, with probability one, there exists a $\mvec{y}^* \in \mathbb{R}^d$ such that $f(\mtx{A}\mvec{y}^*) = f(\mvec{x}^*) = f^*$. 
\end{theorem}

While satisfaction of \eqref{cond: existence_of_y_in_R^d} only depends on $d \geq d_e$, that of \eqref{cond: existence_of_y_in_Y} is determined by the values of both $d$ and $\delta$. For larger values of $d$ and/or $\delta$ the probability that \eqref{cond: existence_of_y_in_Y} is satisfied is higher. One must, on the other hand, be cognisant of the fact that larger values of $d$ --- the dimension of \eqref{eq: REGO} --- and/or $\delta$ --- the half-length of the domain --- demand more computational resources. Therefore, a careful calibration of these two parameters is needed to ensure that \eqref{eq: REGO} is successful for most embeddings, at the same time being capable to converge to the solution within the computational budget. In this regard, our analysis will attempt to answer the following question: What are optimal values of $d$ and $\delta$ such that \eqref{cond: existence_of_y_in_Y} is satisfied with `high' probability?

\section{Characterizing minimizers in the reduced space} \label{sec: characterization}
The analysis of this section focuses on determining the distribution of the random minimizer $\mvec{y}^*$ of $f(\mtx{A}\mvec{y})$, which satisfies $f(\mtx{A}\mvec{y}^*) = f^*$. These results will inform us on the effects of the different parameters on the success of \eqref{eq: REGO} allowing us to estimate the values of $\delta$ and $d$ that are likely to increase the chances of successful recovery of $f^*$.

The following theorem provides a useful characterization of $\mvec{y}^*$. The theorem and its proof were inspired by the proofs of Theorems 2 and 3 in \cite{Wang2016}.  
\begin{theorem}\label{thm: By=z}
	Let \Cref{ass: REGO_fun_eff_dim} hold and let $\mvec{x}_{\top}^*$ and $\mathcal{G}^*$ be defined as in \Cref{def: REGO_choice_of_x^*}. Let $\mtx{A}$ be a $D \times d$ Gaussian matrix. Then, $\mvec{y}^{*} \in \mathbb{R}^d$ satisfies $\mtx{A}\mvec{y}^* \in \mathcal{G}^*$ if and only if
	\begin{equation} \label{eq: By=z}
	\mtx{B} \mvec{y}^* = \mvec{z}^*
	\end{equation}
	where the $d_e \times d$ random matrix $\mtx{B}$ satisfies $ \mtx{B} = \mtx{U}^T \mtx{A}$ and where $\mvec{z}^* \in \mathbb{R}^{d_e} $ is uniquely defined by $\mtx{U}\mvec{z}^* = \mvec{x}_{\top}^*$.
\end{theorem}
\begin{proof}
	Let $\mvec{y}^* \in \mathbb{R}^d$ be such that $\mtx{A}\mvec{y}^* \in \mathcal{G}^*$. First, we establish that 
	\begin{equation}\label{assert: f(Ay)=f(x)_iff_x=UU^TAy}
	\text{$\mtx{A}\mvec{y}^* \in \mathcal{G}^*$ if and only if $\mvec{x}_{\top}^* = \mtx{U}\mtx{U}^T\mtx{A}\mvec{y}^*$.}
	\end{equation}
	
	Suppose that $\mtx{A}\mvec{y}^* \in \mathcal{G}^*$. Then, using the definition of $\mathcal{G}^*$ in \Cref{def: REGO_choice_of_x^*} we can write $\mtx{A}\mvec{y}^* = \mvec{x}_{\top}^* + \mvec{x}_{\perp}$ for some $\mvec{x}_{\perp} \in \mathcal{T}^{\perp}$.
	The orthogonal projection of $\mtx{A}\mvec{y}^*$ onto $\mathcal{T}$ is given by
	$$ \mtx{U}\mtx{U}^T\mtx{A}\mvec{y}^* = \mtx{U}\mtx{U}^T(\mvec{x}_{\top}^* + \mvec{x}_{\perp}) = \mvec{x}_{\top}^*,$$
	where we have used $\mtx{U}\mtx{U}^T \mvec{x}_{\top}^* = \mvec{x}_{\top}^*$ and $\mtx{U}\mtx{U}^T \mvec{x}_{\perp} = \mvec{0}$. 
	
	Conversely, assume that $\mvec{y}^*$ satisfies 
	\begin{equation}\label{eq: x_T^*=UU^TAy}
	\mvec{x}^*_{\top} = \mtx{U}\mtx{U}^T \mtx{A}\mvec{y}^*.
	\end{equation}
	Denote by $\mtx{S}$ the $D \times D$ orthogonal matrix $(\mtx{U} \; \mtx{V})$, where $\mtx{V}$ is defined in \Cref{ass: REGO_fun_eff_dim}. Using \eqref{eq: x_T^*=UU^TAy} and the identity $\mtx{U}\mtx{U}^T + \mtx{V}\mtx{V}^T = \mtx{S} \mtx{S}^T = \mtx{I}_D$, we obtain
	\begin{equation*}
	\mtx{A}\mvec{y}^* = (\mtx{U}\mtx{U}^T + \mtx{V}\mtx{V}^T)\mtx{A}\mvec{y}^* = \mvec{x}_{\top}^* + \mtx{V}\mtx{V}^T\mtx{A}\mvec{y}^*.
	\end{equation*}
	Note that $\mtx{V}\mtx{V}^T\mtx{A}\mvec{y}^*$ lies on $\mathcal{T}^{\perp}$ as it is the orthogonal projection of $\mvec{A}\mvec{y}^*$ onto $\mathcal{T}^{\perp}$, which implies that $\mtx{A}\mvec{y}^* \in \mathcal{G}^*$.
	This completes the proof of \eqref{assert: f(Ay)=f(x)_iff_x=UU^TAy}.
	
	Now we show that \eqref{eq: By=z} and \eqref{eq: x_T^*=UU^TAy} are equivalent. We multiply both sides of $\mvec{x}^*_{\top} = \mtx{U}\mtx{U}^T \mtx{A}\mvec{y}^*$ by $\mtx{S}^T$, and obtain
	\begin{equation} \label{eq: Qx^* = QUU^TAy^*}
	\begin{pmatrix}
	\mtx{U}^T \\
	\mtx{V}^T
	\end{pmatrix} \mvec{x}_{\top}^* = \begin{pmatrix}
	\mtx{U}^T \\
	\mtx{V}^T
	\end{pmatrix} \mtx{U}\mvec{U}^T\mtx{A}\mvec{y}^*.
	\end{equation}
	Since $\mvec{x}_{\top}^*$ is in the column span of $\mtx{U}$, it can be written as $\mvec{x}_{\top}^* = \mtx{U}\mvec{z}^*$
	for some (unique) vector $\mvec{z}^* \in \mathbb{R}^{d_e}$. By substituting the above into \eqref{eq: Qx^* = QUU^TAy^*} we obtain
	$$ \begin{pmatrix}
	\mtx{U}^T\mtx{U}\mvec{z}^* \\
	\mtx{V}^T \mtx{U} \mvec{z}^*
	\end{pmatrix} = \begin{pmatrix}
	\mtx{U}^T\mtx{U} \mtx{U}^T\mtx{A}\mvec{y}^* \\
	\mtx{V}^T \mtx{U} \mtx{U}^T \mtx{A} \mvec{y}^*
	\end{pmatrix}.$$
	This reduces to 
	$$\begin{pmatrix}
	\mvec{z}^* \\
	\mvec{0}
	\end{pmatrix} = \begin{pmatrix}
	\mtx{U}^T\mtx{A}\mvec{y}^* \\
	\mtx{0}
	\end{pmatrix},$$
	where we have used the identities $\mtx{U}^T\mtx{U} = \mtx{I}$ and $\mtx{V}^T\mtx{U} = \mvec{0}$, which follow from \Cref{ass: REGO_fun_eff_dim}.
	To obtain \eqref{eq: x_T^*=UU^TAy} from \eqref{eq: By=z}, multiply \eqref{eq: By=z} by $\mtx{U}$.
\end{proof}

\begin{rem} \label{rem: defined_B_and_z}
	Thereafter, we write $\mtx{B}$ to refer to the $d_e \times d$ Gaussian matrix\footnote{Since $\mtx{U}$ is orthogonal, it follows from \Cref{thm: orthog_inv_of_Gaussian_matrices} that $\mtx{B} = \mtx{U}^T\mtx{A}$ is a Gaussian matrix.} $\mtx{U}^T\mtx{A}$, where $\mtx{A}$ is a $D \times d$ Gaussian matrix and where $\mtx{U}$ is defined in \Cref{ass: REGO_fun_eff_dim}. Furthermore, we write $\mvec{z}^*$ to refer to the $d_e \times 1$ vector that satisfies $\mtx{U}\mvec{z}^* = \mvec{x}_{\top}^*$, where $\mvec{x}_{\top}^*$ is defined in \Cref{def: REGO_choice_of_x^*}. Observe that $\| \mvec{z}^* \| = \| \mvec{x}_{\top}^*\|$ since $\mtx{U}\mvec{z}^* = \mvec{x}_{\top}^*$ and $\mtx{U}$ is orthogonal.
\end{rem}

\begin{corollary}
	Let \Cref{ass: REGO_fun_eff_dim} hold. Let $S^* = \{ \mvec{y}^* : \mtx{A}\mvec{y}^* \in \mathcal{G}^* \}$, where $\mtx{A}$ is a $D \times d$ Gaussian matrix and where $\mathcal{G}^*$ is defined as in \Cref{def: REGO_choice_of_x^*}. Then, the following holds
	\begin{itemize}
		\item If $d = d_e$, then $S^*$ has exactly one element with probability 1.
		\item If $d > d_e$, then $S^*$ has infinitely many elements with probability 1. 
	\end{itemize}
\end{corollary}
\begin{proof}
	It follows from \Cref{thm: By=z} that the set $S^*$ and the set of solutions to $\mtx{B}\mvec{y} = \mvec{z}^*$ coincide. According to \Cref{thm: Wishart_is_nonsingular}, $\mtx{B}\mtx{B}^T$ is positive definite with probability 1, which implies that $\rank(\mtx{B}\mtx{B}^T) = d_e$ with probability 1. Since $\rank(\mtx{B}) = \rank(\mtx{B}\mtx{B}^T)$, $\rank(\mtx{B}) = d_e$ with probability 1. Hence, the linear system $\mtx{B}\mvec{y} = \mvec{z}^*$ almost surely has a solution. If $d = d_e$ the linear system has only one solution. If $d > d_e$ the system is underdetermined and, therefore, has infinitely many solutions.
\end{proof}

\subsection{Choosing a suitable minimizer}
While $S^*$ contains infinitely many solutions if $d > d_e$, it is sufficient that one of these solutions is contained in $\mathcal{Y}$ for \eqref{eq: REGO} to be successful. We proceed further by choosing one particular solution $\mvec{y}^*$ that is easy to analyse and based on the analysis will adjust parameters $\delta$ and $d$ appropriately to ensure that the chosen $\mvec{y}^*$ falls inside the feasible set $\mathcal{Y}$ with high probability. The solutions that are likely to fall inside the feasible domain must be close to the origin. In this regard, we propose two candidates:

\begin{tabularx}{\linewidth-1cm}{@{}XX@{}}
	\noindent
	\begin{equation} \label{eq: minimal_2_norm}
	\begin{aligned}
	\mvec{y}_2^{*}  = \argmin_{\mvec{y} \in \mathbb{R}^d} \;\;& \| \mvec{y} \|_2 \\
	\text{s.t.}\;\; & \mvec{y} \in S^*,
	\end{aligned}
	\end{equation}
	& \noindent
	\begin{equation} \label{eq: minimal_infty_norm}
	\begin{aligned}
	\mvec{y}_{\infty}^* = \argmin_{\mvec{y} \in \mathbb{R}^d} \;\;& \| \mvec{y} \|_{\infty} \\
	\text{s.t.} \;\;& \mvec{y} \in S^*.
	\end{aligned}
	\end{equation}
\end{tabularx}

\noindent
Due to the definition of $\mathcal{Y}$ as a box, the minimizer \eqref{eq: minimal_infty_norm} with the minimal infinity norm is particularly of interest. Since $\mvec{y}_{\infty}^*$ has the smallest infinity norm among all solutions in $S^*$, knowledge of $\mvec{y}_{\infty}^*$ would allow us to choose the smallest possible $\mathcal{Y}$ while ensuring that \eqref{eq: REGO} is successful. Despite this convenient fact, we found that it is more difficult to study $\mvec{y}_{\infty}^*$ and have decided to investigate $\mvec{y}_{2}^*$ instead. 
\begin{rem}
	For $d = d_e$, $\mvec{y}_2^* = \mvec{y}_{\infty}^*$ because $S^*$ contains only one element.
\end{rem}
\begin{lemma} \label{lemma: REGO_successful_if_y_2_is_in_Y}
	Let \Cref{ass: REGO_fun_eff_dim} hold and let $\mathcal{G}^*$ and $\mvec{y}^*_{2}$ be defined as in \Cref{def: REGO_choice_of_x^*} and \eqref{eq: minimal_2_norm}, respectively.
	Problem \eqref{eq: REGO} is successful in the sense of \Cref{def: successful_REGO} if $\mvec{y}^*_{2} \in \mathcal{Y}$.
\end{lemma}
\begin{proof}
	Assume that $\mvec{y}_{2}^* \in \mathcal{Y}$. Then, $\mvec{y}_{2}^*$ is a feasible solution of \eqref{eq: REGO}. By the definitions of $\mvec{y}_{2}^*$ and $S^*$, we also have $\mtx{A}\mvec{y}_{2}^* \in \mathcal{G}^*$; this implies that $f(\mtx{A}\mvec{y}_{2}^*) = f^*$. Hence, \eqref{eq: REGO} is successful by \Cref{def: successful_REGO}.
\end{proof}
\begin{corollary} \label{cor: y_2_explicit_formula}
	Let \Cref{ass: REGO_fun_eff_dim} hold. Problem $\eqref{eq: minimal_2_norm}$ has a unique solution given by 
	\begin{equation} \label{eq: Euclidean_norm_sol_explicit_formula}
		\mvec{y}_2^* = \mtx{B}^T(\mtx{B}\mtx{B}^T)^{-1}\mvec{z}^*.
	\end{equation}
\end{corollary}
\begin{proof}
	It follows from \Cref{thm: By=z} that the solution(s) of \eqref{eq: minimal_2_norm} must be equal to the solution(s) of the following problem
	\begin{equation*} \label{eq: minimal_2_norm_By=z}
	\begin{aligned}
	\min \;\;& \| \mvec{y} \|_2^2 \\
	\text{s.t.}\;\; & \mtx{B}\mvec{y} = \mvec{z}^*,
	\end{aligned}
	\end{equation*}
	which has the solution \eqref{eq: Euclidean_norm_sol_explicit_formula}. 
\end{proof}

\subsection{Distribution of minimal Euclidean norm minimizer}
The present section derives the distribution of $\mvec{y}_2^*$ and its Euclidean norm. 
\subsubsection*{The distribution of the Euclidean norm of $\mvec{y}_2^*$}
\begin{theorem} \label{thm: x^*_T/y^*_2_follow_chi_square}
	Let \Cref{ass: REGO_fun_eff_dim} hold and let $\mvec{x}_{\top}^*$ and $\mvec{y}^*_2$ be defined as in \Cref{def: REGO_choice_of_x^*} and \eqref{eq: minimal_2_norm}, respectively. Then, $\mvec{y}^*_2$ satisfies
	$$ \frac{\| \mvec{x}_{\top}^* \|^2_2}{\| \mvec{y}^*_2 \|^2_2} \sim \chi^2_{d-d_e+1}.$$
\end{theorem}
\begin{proof}
	The result almost immediately follows from \Cref{cor: y_2_explicit_formula} and \Cref{lemma: inverse_min_norm_solution_follow_chi-square}. These yield
	$$ \frac{\|\mvec{z}^*\|^2}{\| \mvec{y}^*_2 \|^2} \sim \chi^2_{d-d_e+1},$$
	The result is implied by $\| \mvec{z}^* \| = \|\mvec{x}_{\top}^*\|$ (see \Cref{rem: defined_B_and_z}).
\end{proof}
The above result is equivalent to saying that $\|\mvec{y}^*_2\|^2/\| \mvec{x}^*_{\top} \|^2$ follows the inverse chi-squared distribution with $d-d_e+1$ degrees of freedom (see \Cref{def: inv_chi-squared}). The theorem reveals a linear dependence of $\|\mvec{y}_2^*\|$ on $\|\mvec{x}^*_{\top} \|$; larger values of $\|\mvec{x}^*_{\top} \|$ contribute to the increase in the likelihood of $\mvec{y}_2^*$ being further away from the origin. 
The theorem also suggests that $\|\mvec{y}_2^*\|$ is independent of $D$ as long as $\| \mvec{x}^*_{\top} \|$ is fixed.  

\begin{corollary} \label{cor: prob_of_y_2^*<delta}
	Let \Cref{ass: REGO_fun_eff_dim} hold. Let $\mvec{x}_{\top}^*$ and $\mvec{y}_2^*$ be defined as in \Cref{def: REGO_choice_of_x^*} and \eqref{eq: minimal_2_norm}, respectively. Then,
	$$ \prob[\| \mvec{y}_2^* \|_2 \leq \delta] = \prob\bigg[\chi^2_{d-d_e+1} \geq \frac{\| \mvec{x}^*_{\top} \|_2^2}{\delta^2}\bigg] $$
	for any $\delta > 0$. 
\end{corollary}
\begin{proof}
	For any $\epsilon > 0$, we have
	$$ \prob\bigg[ \|\mvec{y}^*_2\|_2 \leq \frac{\|\mvec{x}_{\top}^*\|_2}{\epsilon} \bigg] =  \prob\bigg[ \frac{\| \mvec{x}_{\top}^* \|^2_2}{\| \mvec{y}^*_2 \|^2_2} \geq \epsilon^2 \bigg] = \prob[ \chi^2_{d-d_e+1} \geq \epsilon^2 ],$$
	where the second equality follows from \Cref{thm: x^*_T/y^*_2_follow_chi_square}. By letting $\epsilon = \|\mvec{x}_{\top}^*\|_2/\delta$, we obtain the result. 
\end{proof}

\begin{corollary} \label{cor: exp_val_of_norm_y_2^*}
	Let \Cref{ass: REGO_fun_eff_dim} hold and let $\mvec{x}_{\top}^*$ and $\mvec{y}^*_2$ be defined as in \Cref{def: REGO_choice_of_x^*} and \eqref{eq: minimal_2_norm}, respectively. Provided that $d-d_e > 1$ we have
	\begin{equation}\label{eq: expected_value_of_norm_of_y_2^*}
	\mathbb{E}[\| \mvec{y}_2^* \|^2] = \frac{\| \mvec{x}_{\top}^* \|^2}{d-d_e-1}.
	\end{equation}
\end{corollary}
\begin{proof}
	Let $W$ follow the inverse chi-squared distribution with $d-d_e +1$ degrees of freedom. Then, $W \stackrel{law}{=} \| \mvec{y}_2^* \|^2/\| \mvec{x}_{\top}^* \|^2$ by \Cref{thm: x^*_T/y^*_2_follow_chi_square}. By applying \Cref{lemma: expectation_inverse_chi_squared}, we obtain
	$$  \mathbb{E}[\| \mvec{y}_2^* \|^2] = \mathbb{E}[\| \mvec{x}_{\top}^* \|^2 W] = \frac{\| \mvec{x}_{\top}^* \|^2}{d-d_e-1}$$
	for $d - d_e > 1$. 
\end{proof}
The expected value in \eqref{eq: expected_value_of_norm_of_y_2^*} is inversely proportional to $d-d_e$. In other words, for a fixed $d_e$,  larger values of the dimension of the embedding subspace bring $\mvec{y}_2^*$ closer to the origin. This observation indicates that the increase in $d$ allows us to decrease $\delta$ whilst the probability of $\mvec{y}_2^* \in \mathcal{Y}$ is kept constant. 

\subsubsection*{The probability density function}
The following theorem derives the probability density function of $\mvec{y}_2^*$. 

\begin{theorem} \label{thm: pdf_of_y^*}
	Let \Cref{ass: REGO_fun_eff_dim} hold and let $\mvec{x}_{\top}^*$ and $\mvec{y}^*_2$ be defined as in \Cref{def: REGO_choice_of_x^*} and \eqref{eq: minimal_2_norm}, respectively. Then, the probability density function of $\mvec{y}^*_2$ is given by
	\begin{equation*} \label{eq: pdf_of_y}
	g^*(\mvec{y}) = \pi^{-d/2} \bigg(\frac{\Gamma(d/2)}{\Gamma(n/2)}\bigg) \bigg(\frac{\| \mvec{x}_{\top}^* \|}{\sqrt{2}}\bigg)^n (\mvec{y}^T\mvec{y})^{-(n+d)/2} e^{-\|\mvec{x}_{\top}^*\|^2/(2\mvec{y}^T\mvec{y})},
	\end{equation*}
	where $n = d-d_e+1$.
\end{theorem}
\begin{proof}
	\Cref{cor: y_2_explicit_formula} and \Cref{lemma: pdf_of_y_general_case} imply that the p.d.f. of $\mvec{y}_2^*$ is given by
	$$ g^* (\mvec{y}) = \pi^{-d/2} \bigg(\frac{\Gamma(d/2)}{\Gamma(n/2)}\bigg) \bigg(\frac{\| \mvec{z}^* \|}{\sqrt{2}}\bigg)^n   (\mvec{y}^T\mvec{y})^{-(n+d)/2} e^{-\|\mvec{z}^*\|^2/(2\mvec{y}^T\mvec{y})}.$$
	By using the equation $\| \mvec{z}^* \| = \|\mvec{x}_{\top}^*\|$, we obtain the desired result.
\end{proof}

\begin{figure}
	\centering
	\includegraphics[scale = 0.5]{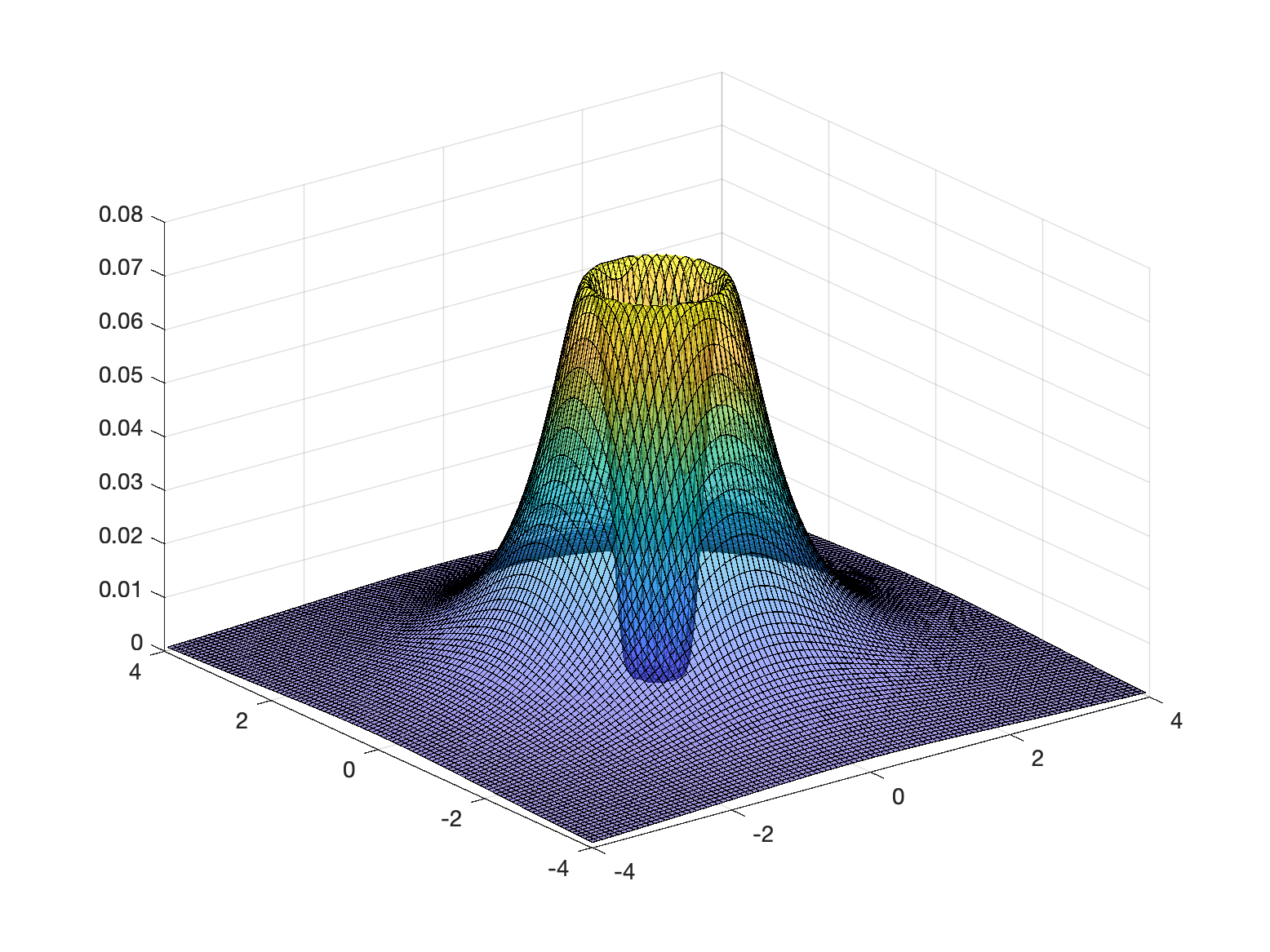}
	\caption{The illustration of the p.d.f. of $\mvec{y}_2^*$ for $d = 2$, $k = 2$ and $\mvec{x}_{\top}^* = [1 \; 1]^T$.}
	\label{fig: pdf_of_y_2^*}
\end{figure}
\Cref{fig: pdf_of_y_2^*} illustrates the p.d.f. of two-dimensional $\mvec{y}_2^*$. The shape of the p.d.f.~resembles a volcano with the mass concentrated at a certain distance from the origin suggesting that $\mvec{y}_2^*$ is unlikely to be neither too close to, nor too distant from the origin. We also note that the p.d.f.~is independent of $D$.

\section{Bounding the success of the reduced problem} \label{sec: bounding_success}
This section is the culmination of this paper's analysis. Based on the results established earlier we derive a bound for the probability of success of \eqref{eq: REGO}.

The following theorem presents a notable connection between the success of \eqref{eq: REGO} and the chi-squared distribution. 
\begin{theorem}\label{thm: prob_REGO_is_successful}
	Let \Cref{ass: REGO_fun_eff_dim} hold. Then, for any $\delta > 0$, we have
	\begin{equation} \label{eq: in_thm_prob_REGO_is_successful}
		\prob[\eqref{eq: REGO} \text{ is successful}\,] \geq \prob\bigg[\chi^2_{d-d_e+1} \geq \frac{ \min_{\mvec{x}^* \in \mathcal{G}} \|  \mvec{x}^* \|_2^2}{\delta^2}\bigg].
	\end{equation}
\end{theorem}
\begin{proof}
	Note the following relationship between the probabilities:
	\begin{align}\label{ineq: RP_succ_series_of_bounds}
	\prob[\eqref{eq: REGO} \text{ is successful}] \geq \prob[\mvec{y}^*_{2} \in \mathcal{Y}] =  \prob[\|\mvec{y}^*_{2}\|_{\infty} \leq \delta] \geq \prob[\|\mvec{y}^*_{2}\|_{2} \leq \delta],
	\end{align}
	where the first inequality follows from \Cref{lemma: REGO_successful_if_y_2_is_in_Y} and where the second inequality is implied by $\| \mvec{y}^*_{2}\|_{\infty}  \leq \|\mvec{y}^*_{2}\|_{2}$. By applying 
	\Cref{cor: prob_of_y_2^*<delta} to the last probability in \eqref{ineq: RP_succ_series_of_bounds} and using the definition of $\mvec{x}_{\top}^*$ given in \Cref{def: REGO_choice_of_x^*}, we obtain 
	\begin{equation}\label{eq: chi^2> norm_of_UU^Tx^*/delta^2}
		\prob[\eqref{eq: REGO} \text{ is successful}] \geq \prob\bigg[\chi^2_{d-d_e+1} \geq \frac{ \| \mtx{U}\mtx{U}^T \mvec{x}^* \|_2^2}{\delta^2}\bigg]
	\end{equation}
	for any $\delta > 0$ and any $\mvec{x}^* \in \mathcal{G}$ such that $ \| \mtx{U}\mtx{U}^T \mvec{x}^* \|_2 \neq 0$. Note that \eqref{eq: chi^2> norm_of_UU^Tx^*/delta^2} also holds for $\| \mtx{U}\mtx{U}^T \mvec{x}^* \|_2 = 0$ since, in this case, \eqref{eq: REGO} is successful with probability 1 (see the discussion preceding \Cref{def: REGO_choice_of_x^*}). Hence, \eqref{eq: chi^2> norm_of_UU^Tx^*/delta^2} holds for any $\mvec{x}^* \in \mathcal{G}$, which then implies 
	$$ \prob[\eqref{eq: REGO} \text{ is successful}] \geq \max_{\mvec{x}^* \in \mathcal{G}} \prob\bigg[\chi^2_{d-d_e+1} \geq \frac{ \| \mtx{U}\mtx{U}^T \mvec{x}^* \|_2^2}{\delta^2}\bigg] = \prob\bigg[\chi^2_{d-d_e+1} \geq \frac{ \min_{\mvec{x}^* \in \mathcal{G}} \| \mtx{U}\mtx{U}^T \mvec{x}^* \|_2^2}{\delta^2}\bigg],$$
	where the equality follows from the fact that the tail distribution $\prob[X>x]$ of any random variable $X$ is a monotonically decreasing function in $x$. 
	
	In what follows, we show that $\min_{\mvec{x}^* \in \mathcal{G}} \| \mtx{U}\mtx{U}^T \mvec{x}^* \|_2^2 = \min_{\mvec{x}^* \in \mathcal{G}} \| \mvec{x}^* \|_2^2$. Define sets $\mathcal{Z} = \{ \mvec{z} \in \mathbb{R}^d : \mtx{U}\mvec{z} = \mtx{U}\mtx{U}^T\mvec{x}^*, \mvec{x}^* \in \mathcal{G} \}$ and $\mathcal{S} = \{ \mtx{U}\mvec{z} + \mvec{V}\mvec{c} : \mvec{z} \in \mathcal{Z}, \mvec{c} \in \mathbb{R}^{D-d_e} \}$, where $\mtx{V}$ is defined in \Cref{def: REGO_choice_of_x^*}. First, we establish that $\mathcal{G} = \mathcal{S}$ by showing that $\mathcal{G} \subseteq \mathcal{S}$ and that $\mathcal{S} \subseteq \mathcal{G}$. 
	
	Let $\mvec{x}^* \in \mathcal{G}$. We can write $\mvec{x}^* = \mtx{U}\mtx{U}^T \mvec{x}^* + \mtx{V}\mtx{V}^T\mvec{x}^*$ since $\mtx{U}\mtx{U}^T + \mtx{V}\mtx{V}^T = \mtx{I}$. Let $\mvec{z} = \mtx{U}^T\mvec{x}^*$ and $\mvec{c} = \mtx{V}^T\mvec{x}^*$ and note that $\mvec{z} \in \mathcal{Z}$ and $\mvec{c} \in \mathbb{R}^{D-d_e}$. Hence, $\mvec{x}^* \in \mathcal{S}$, which proves that $\mathcal{G} \subseteq \mathcal{S}$. 
	
	Let $\mvec{x}^* \in \mathcal{S}$. Then, $\mvec{x}^* = \mtx{U}\mvec{z} + \mtx{V}\mvec{c}$ for some $\mvec{z}\in\mathcal{Z}$ and $\mvec{c} \in \mathbb{R}^{D-d_e}$. We have
	$$ f(\mvec{x}^*) = f(\mtx{U}\mvec{z}+\mtx{V}\mvec{c}) = f(\mtx{U}\mvec{z}) = f(\mtx{U}\mtx{U}^T\mvec{x}^*) = f(\mvec{x}_{\top}^*) = f^*,$$
	where the second equality follows from the assumption that $f$ has low effective dimensionality and the fact that $\mtx{V}\mvec{c} \in \mathcal{T}^{\perp}$, the fourth equality follows from the definition of $\mvec{x}_{\top}^*$ (given in \Cref{def: REGO_choice_of_x^*}) and the last equality follows from \eqref{eq: f^*=f(x_top^*)}. Hence, by definition of $\mathcal{G}$, $\mvec{x}^* \in \mathcal{G}$. This proves that $\mathcal{S} \subseteq \mathcal{G}$.
	
	Finally, we have
	\begin{align*}
		\min_{\mvec{x}^* \in \mathcal{G}} \| \mvec{x}^* \|_2^2 = \min_{\mvec{x}^* \in \mathcal{S}} \| \mvec{x}^* \|_2^2 & = \min_{\mvec{z} \in \mathcal{Z}, \, \mvec{c} \in \mathbb{R}^{D-d_e}} \| \mtx{U}\mvec{z} + \mtx{V}\mvec{c} \|_2^2 & \text{(since $\mathcal{G} = \mathcal{S}$ and by definition of $\mathcal{S}$)}  \\
		 & = \min_{\mvec{z} \in \mathcal{Z}, \, \mvec{c} \in \mathbb{R}^{D-d_e}}  \| \mtx{U} \mvec{z}\|_2^2 + \| \mtx{V}\mvec{c} \|_2^2 & \text{(since $\mtx{U}^T\mtx{V} = \mtx{V}^T\mtx{U} = \mtx{0}$)} \\
		 & = \min_{\mvec{z} \in \mathcal{Z}, \, \mvec{c} \in \mathbb{R}^{D-d_e}}  \| \mtx{U}\mvec{z} \|^2_2 + \| \mvec{c} \|_2^2 & \text{(since $\mtx{V}$ is orthogonal)}\\
		 & = \min_{\mvec{z}\in\mathcal{Z}} \| \mtx{U}\mvec{z} \|_2^2 +  \min_{\mvec{c}\in \mathbb{R}^{D-d_e}} \| \mvec{c} \|_2^2 \\
		 & = \min_{\mvec{z}\in\mathcal{Z}} \| \mtx{U} \mvec{z} \|_2^2 + 0  \\
		 & = \min_{\mvec{x}^* \in \mathcal{G}} \| \mtx{U}\mtx{U}^T \mvec{x}^* \|_2^2 & \text{(by definition of $\mathcal{Z}$)}
	\end{align*}

\end{proof}
Using \Cref{thm: prob_REGO_is_successful}, one can now bound the success of \eqref{eq: REGO} by applying any tail bound on the chi-squared distribution. We use the bound derived in \Cref{lemma: chi-sq_tail_bound}. 
\begin{corollary} \label{cor: REGO_successful_lower_bound}
	Let \Cref{ass: REGO_fun_eff_dim} hold and let $\mu = \min_{\mvec{x}^* \in \mathcal{G}} \|  \mvec{x}^* \|_2$. Then, for any $\delta > 0$, we have
	\begin{equation}\label{ineq: REGO_successful_bound}
		\prob[\eqref{eq: REGO} \text{ is successful}\,] \geq 1 -  C(n) \bigg(1+ \frac{n}{2}e^{-\mu^2/(2\delta^2)}\bigg)\bigg(\frac{\mu}{\sqrt{2}\delta}\bigg)^{n},
	\end{equation}
	where $n = d-d_e+1$ and 
	$$ C(n) = \frac{4}{n(n+2)\Gamma\big( n/2 \big)}.$$
\end{corollary}
\begin{proof}
	\Cref{lemma: chi-sq_tail_bound} implies that
	\begin{equation}\label{eq: prob_chi^2 > eps^2 > 1-C(n)... }
		\prob[ \chi^2_{n} \geq \epsilon^2 ] \geq 1 - C(n)\bigg(1+\frac{n}{2}e^{-\epsilon^2/2}\bigg)(\epsilon^2/2)^{n/2}
	\end{equation}
	for any $\epsilon > 0$. By letting $\epsilon = \mu/\delta$ and applying \eqref{eq: prob_chi^2 > eps^2 > 1-C(n)... } to \eqref{eq: in_thm_prob_REGO_is_successful}, we obtain the wished bound. 
\end{proof}

Let $R^*$ denote the right hand side of \eqref{ineq: REGO_successful_bound}. First, we note that $R^*$ is a function of $\mu/\delta$ and $d-d_e$. The bound reveals a linear relationship between $\mu$ and $\delta$; scaling $\mu$ and $\delta$ by the same factor does not affect the value of $R^*$. Furthermore, observe that for smaller values of $\mu$ or larger values of $\delta$, $R^*$ is closer to 1. Numerical experiments show that for large values of $n$ and/or $\mu/\delta$, the bound \eqref{ineq: REGO_successful_bound} is less tight; this is also signified by the asymptotic behaviour of $R^*$, $R^* \rightarrow -\infty$ monotonically as $\mu/\delta \rightarrow \infty$ making the bound useless for large enough $\mu/\delta$. 

It is remarkable that $R^*$ has no dependence on $D$, the dimension of the original optimization problem. This implies that larger $D$ does not diminish the success of the reduced problem as long as $\mu$ and $d_e$ are unchanged. Dependence of $R^*$ on $d-d_e$ indicates that the success is determined by the value of $d$ relative to $d_e$ and not so much by the individual values of $d$ and $d_e$. Larger (smaller) values of $d$ with respect to $d_e$ require smaller (larger) $\delta$ if $R^*$ is kept constant; knowing this fact is crucial when initializing values of $d$ and $\delta$ in practice. It displays a convenient interplay between $d$ and $\delta$ allowing more flexibility in choosing one versus another. 

\paragraph{Previous bounds.}
One can derive similar bounds for the success of \eqref{eq: REGO} by bounding $ \prob[\|\mvec{y}^*_{2}\| \leq \delta]$ in \eqref{ineq: RP_succ_series_of_bounds} using the Cauchy-Schwarz inequality. Since $\mvec{y}_2^* = \mtx{B}^T(\mtx{B}\mtx{B}^T)^{-1}\mvec{z}^* $, we have
$$ \| \mvec{y}^*_2 \| \leq \|\mtx{B}^T(\mtx{B}\mtx{B}^T)^{-1} \| \cdot \|\mvec{z}^* \|.$$
By using the fact that $\|\mtx{B}^T(\mtx{B}\mtx{B}^T)^{-1} \| = 1/s_{\min}(\mtx{B}^T)$, where $s_{\min}(\mtx{B}^T)$ denotes the smallest singular value of $\mtx{B}^T$, we obtain
$$ \prob[\|\mvec{y}^*_{2}\| \leq \delta] \geq \prob\bigg[ \frac{\| \mvec{z}^* \|}{s_{\min}(\mtx{B}^T)}  \leq \delta \bigg].$$
We can now use any suitable tail bound for the smallest singular value of the Gaussian matrix to bound the latter probability. 

Wang et al.~\cite{Wang2016}, by applying the above technique and the result in \cite{Edelman1988} to bound the singular value, derived the following bound
\begin{equation*}\label{ineq: Wang_bound}
\prob[\eqref{eq: REGO} \text{ is successful}] \geq 1 - \frac{\mu \sqrt{d_e}}{\delta}.
\end{equation*} 
Their derivation is predicated on the assumptions that $d = d_e$ and that $\mathcal{T}$ is spanned by the standard basis vectors. In \cite{Sanyang2016}, Sanyang and Kab\'{a}n extended Wang et al.'s bound to any $\delta$ satisfying $\delta> \| \mvec{x}^*_{\top} \|/(\sqrt{d}-\sqrt{d_e})$. Using the bound in \cite{Davidson2001} for $s_{\min}(\mtx{B}^T)$ they showed that
\begin{equation*} 
\prob[\eqref{eq: REGO} \text{ is successful}] \geq 1 - e^{-(\sqrt{d}-\sqrt{d_e}-\mu/\delta)^2/2}
\end{equation*}
One can also use Rudelson and Vershynin's bound in \cite[Theorem 1.1]{Rudelson2009} to obtain 
$$\prob[\eqref{eq: REGO} \text{ is successful}] \geq 1 - \bigg(\frac{C\mu}{\delta (\sqrt{d} - \sqrt{d_e-1}) }\bigg)^{d-d_e+1} - e^{-cd},$$
where $C,c > 0$ are absolute constants. This bound shows dependence of the probability on the difference $d-d_e$, which is also manifest in our bound. The Rudelson and Vershynin's bound cannot be used for practical purposes due to the unknown $C$ and $c$; we require explicit bounds to define the size of $\mathcal{Y}$.

Unlike the bounds of Wang et al. \cite{Wang2016} and Sanyang and Kaban \cite{Sanyang2016}, \Cref{cor: prob_of_y_2^*<delta} is applicable to any $d \geq d_e$ and an arbitrary subspace $\mathcal{T}$. Moreover, using the exact distribution of $\| \mvec{y}_2^* \|$ given in \Cref{cor: prob_of_y_2^*<delta}, we circumvent the application of the intermediate Cauchy-Schwarz and bound the distribution of $\| \mvec{y}_2^* \|$ directly. 

\paragraph{Affine random embeddings.} \label{par: affine_random_embeddings}
It is not difficult to extend \eqref{eq: REGO} to affine random subspace embeddings. In the affine case, we replace $\mvec{x}$ by $\mtx{A}\mvec{y} + \mvec{p}$, where $\mvec{p} \in \mathbb{R}^D$ is a fixed point. The reduced optimization problem is then given by
\begin{equation*} \label{reduced problem affine}
\begin{aligned}
\min \;\; & f(\mtx{A}\mvec{y} + \mvec{p}) \\
\text{subject to} \;\; & \mvec{y} \in \mathcal{Y},
\end{aligned}
\end{equation*}
The results that apply to the linear embeddings also apply to the affine embeddings after minor adjustments.
\Cref{thm: Wang_theorem}, for example, can be easily extended to the affine case to show that the intersection between $\mvec{p} + \range(\mtx{A})$ and $\mathcal{G}$ takes place with probability~1 if $d \geq d_e$. The affine version of the results are provable with the same assumptions except for a minor alteration in \Cref{def: REGO_choice_of_x^*}: the condition $\mvec{x}_{\top}^* \neq \mvec{0}$ changes to $\mvec{x}_{\top}^* \neq \mvec{p}$. To obtain the affine versions of \Cref{thm: x^*_T/y^*_2_follow_chi_square} and \Cref{thm: pdf_of_y^*}, replace $\mvec{x}_{\top}^*$ with $\mvec{x}_{\top}^* - \mvec{p}_{\top}$, where $\mvec{p}_{\top} = \mtx{U}\mtx{U}^T\mvec{p}$ is the orthogonal projection of $\mvec{p}$ onto $\mathcal{T}$. For the affine versions of \Cref{thm: prob_REGO_is_successful} and \Cref{cor: REGO_successful_lower_bound}, replace $\min_{\mvec{x}^* \in \mathcal{G}} \| \mvec{x}^* \|$ with $\min_{\mvec{x}^* \in \mathcal{G}} \| \mvec{x}^* - \mvec{p} \|$.

%
%
%

\section{Numerical experiments} \label{sec: Numerical_experiments}

\subsection{Choices of \eqref{eq: REGO} parameters} \label{sec: exp_setup}
\begin{figure}[!t]
	\centering
	\includegraphics[scale = 0.65]{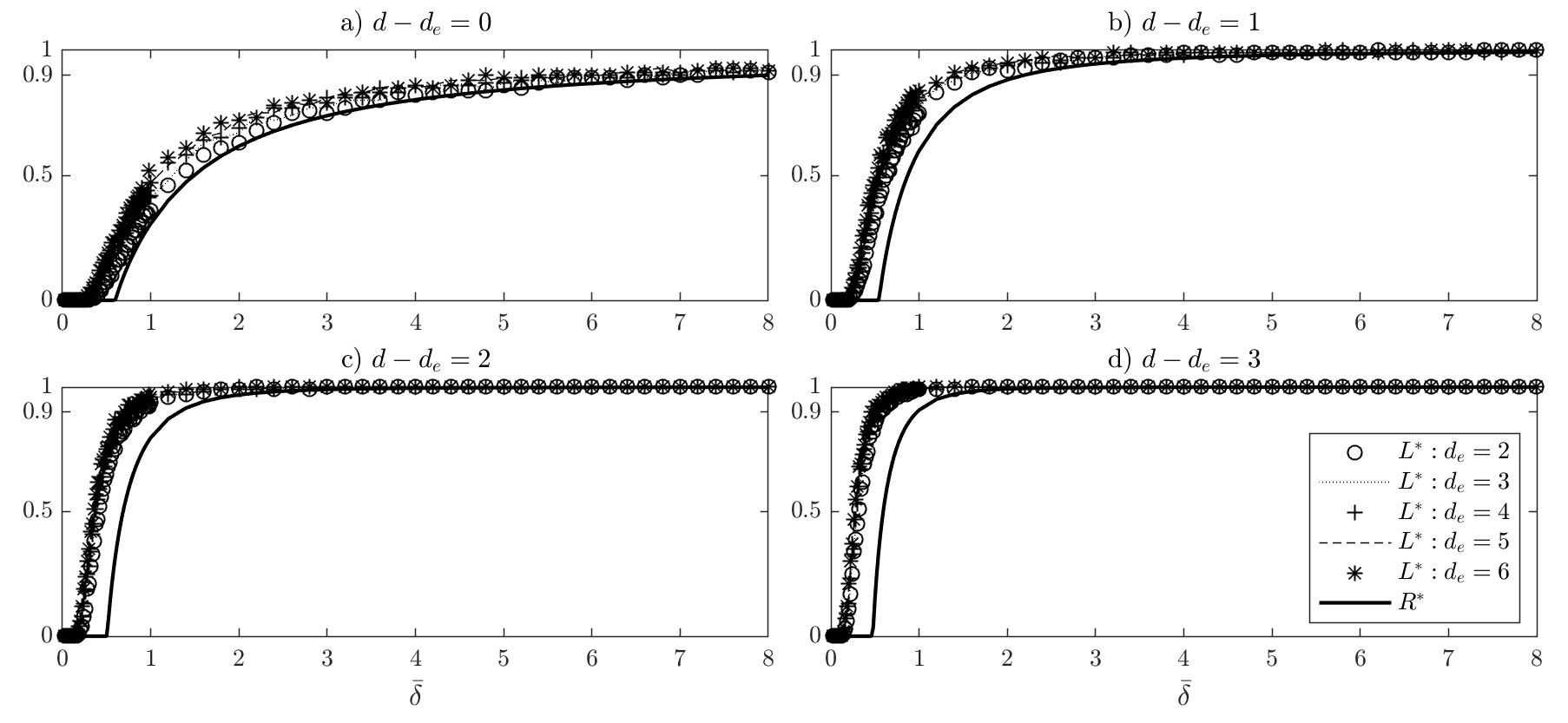}
	\caption{The four plots depict the function $R^*(\bar{\delta})$ and the estimates of $L^*(\bar{\delta})$; each plot corresponds to a particular value of $d-d_e \in \{ 0,1,2,3 \}$. Each plot contains estimates of $L^*(\bar{\delta})$ for $d_e = 2,3,4,5,6$.
	}
	\label{fig: par_est}
\end{figure}
The present section aims to test numerically the quality of the bound \eqref{ineq: REGO_successful_bound}. We will also use the results of this section to select suitable pairs of parameters $d$ and $\delta$ for \eqref{eq: REGO} in the numerical experiments later. 

Suppose that we are given a function $f$ satisfying \Cref{ass: REGO_fun_eff_dim} with the set of global minimizers $\mathcal{G}$ consisting of only one connected component. Let $\mvec{x}_{\top}^*$ for $f$ be defined as in \Cref{def: REGO_choice_of_x^*} and $\mvec{z}$ be defined by the equation $\mtx{U}\mvec{z} = \mvec{x}^*_{\top}$. We also define $\mu := \min_{\mvec{x}^* \in \mathcal{G}} \| \mvec{x}^* \|$ and note that $\mu= \| \mvec{x}_{\top}^* \|$. 

We test \eqref{ineq: REGO_successful_bound} for $f$ by contrasting the left-hand side of \eqref{ineq: REGO_successful_bound} (denoted by $L^*$) to its right-hand side (denoted by $R^*$). We compare $L^*$ and $R^*$ for four different values of $d - d_e$, namely, 0, 1, 2 and 3. For each value of $d-d_e$, we express $R^*$ as a function of $\bar{\delta} := \delta/\mu$ and using its closed form we plot $R^*$ for $\bar{\delta} \in [0.02, 10]$. We do not have a closed form expression for $L^*$, but we can approximate it numerically. In what follows, we describe how this could be done. We start by writing
\begin{equation} \label{eq: L^* = prob[exists y^*]}
\begin{aligned}
L^* := \prob[\eqref{eq: REGO} \text{ is successful}] = \prob[ \text{$\exists\mvec{y}\in [-\delta, \delta]^d : \mtx{A}\mvec{y} \in \mathcal{G}$} ] & = \prob[ \text{$\exists\mvec{y}\in [-\delta, \delta]^d : \bar{\mtx{B}}\mvec{y} = \mvec{z}$} ] \\
& = \prob[ \text{$\exists\mvec{y}\in [-\bar{\delta}, \bar{\delta}]^d : \bar{\mtx{B}}\mvec{y} = \bar{\mvec{z}}$} ],
\end{aligned}
\end{equation}
where $\bar{\mtx{B}}$ denotes a $d_e \times d$ Gaussian matrix and $\bar{\mvec{z}} = \mvec{z}/\mu$. Here, the second equality follows from \Cref{def: successful_REGO}, and the third equality follows from \Cref{thm: By=z} and the fact that $\mathcal{G}$ has only one connected component. Note that $\| \bar{\mvec{z}} \| = 1$ since $\|\mvec{z}\| = \| \mvec{x}_{\top}^* \| = \mu$ (see \Cref{rem: defined_B_and_z}). We assign $\bar{\mvec{z}}$ to a random vector with unit norm and keep $\bar{\mvec{z}}$ fixed throughout the experiment\footnote{Note that the results of the experiment are invariant of the choice of $\bar{\mvec{z}}$ as long as its norm is fixed. Let $\mvec{z}_1$ and $\mvec{z}_2$ be two fixed vectors with unit norm. Consider two systems: $\mtx{B}\mvec{y} = \mvec{z}_1$ and $\mtx{B}\mvec{y} = \mvec{z}_2$. Note that $\mvec{z}_2$ can be written as $\mtx{Q}\mvec{z}_1$ for some orthogonal $\mtx{Q} \in \mathbb{R}^{d_e \times d_e}$. Then, the second system becomes $\mtx{Q}^T \mtx{B} \mvec{y} = \mvec{z}_1$ and this generates vectors $\mvec{y}$ with the same distribution as the first system since $\mtx{Q}^T\mtx{B}$ is also Gaussian.}. For each $\bar{\delta} \in [0.02, 10]$, we generate a thousand Gaussian matrices $\bar{\mtx{B}}$ and estimate the latter probability in \eqref{eq: L^* = prob[exists y^*]} as the proportion of instances for which the statement under the probability is true. Unlike for $R^*$, $L^*$ depends on individual values of $d$ and $d_e$. We plot the estimates for $L^*$ for the following values of $d-d_e$: 0, 1, 2, 3 and, in each plot, we repeat the experiment for $d_e = 2,3,4,5,6$. The plots are presented in \Cref{fig: par_est}. 
\paragraph{Numerical findings.}
The plots in \Cref{fig: par_est} --- confirming the conclusions of \Cref{cor: REGO_successful_lower_bound} --- suggest that the variation in success of \eqref{eq: REGO} is mainly determined by the value of $d-d_e$; the larger is the difference, the higher is the probability of success of \eqref{eq: REGO} for a given $\bar{\delta}$. These curves, being independent of $\mu$, can be used to find suitable $\bar{\delta}$ for any problem for the corresponding values of $d-d_e$; the size of the $\mathcal{Y}$ box, $\delta$, can then be set to $M \bar{\delta}$ if an upper bound $M$ on $\mu$ is known.

\paragraph{Choosing $d$ and $\delta$ in practice.}
When it comes to the numerical application of \eqref{eq: REGO} in practice, initialization of parameters $d$ and $\delta$ might be problematic. From the theoretical discussions above we learned that the parameters $d$ and $\delta$ must be defined based on  $d_e$ and $\mu$, the values of which are typically unknown in practice, for example, for black-box functions. We circumvent this issue by estimating $d_e$ and $\mu$ rather than trying to calculate their exact values; note that all we need is an upper bound $d$ on $d_e$. The parameter $d_e$ or an upper bound may be known from prior studies or can be found with active subspace identification methods (see, e.g., \cite{Constantine2015}); these use gradients of $f$ to estimate $d_e$. 

Estimating $\mu$ can be a harder task. A rough estimate for $\mu$ can be obtained if the search in the original space is restricted to a certain domain; a trivial upper bound in this case is given by the maximum distance between the origin and the boundary of the domain. The search domain that is commonly imposed to practically solve unconstrained optimization problems is box constraints, such as $\mathcal{X} = [-1,1]^D$ for which $\mu \leq \sqrt{D}$. In \Cref{app: more_experiments}, we test REGO assuming that $\sqrt{D}$ is the best bound known for $\mu$. To compensate for unknown $\mu$, one could also try increasing $\delta$ or $d$ gradually to explore larger regions in $\mathbb{R}^D$. \label{par: choosing_d_and_delta_in_practice}

\subsection{Testing REGO with state-of-the-art global solvers}
\paragraph{Algorithms.}
\begin{algorithm}[!t]
	\centering
	\caption{Random Embeddings for Global Optimization (REGO) applied to \eqref{eq: GO}.}
	\label{alg: REGO}
	\begin{algorithmic}[1]
		\item Initialise $d$ and $\delta$ and define $\mathcal{Y} = [-\delta, \delta]^d$
		\item Generate a $D\times d$ Gaussian matrix $\mtx{A}$
		\item Apply a global optimization solver (e.g. BARON, DIRECT, KNITRO) to \eqref{eq: REGO} until a termination criterion is satisfied, and define $\mvec{y}_{min}$ to be the generated (approximate) solution of \eqref{eq: REGO}.
		\item Reconstruct $\mvec{x}_{min} = \mtx{A} \mvec{y}_{min}$
	\end{algorithmic}
\end{algorithm}	
The algorithm for the random embeddings method named REGO (Random Embeddings for Global Optimization) is outlined in \Cref{alg: REGO}. Below, we give the descriptions of the three state-of-the-art solvers we use to test REGO.  

DIRECT(\cite{Gablonsky2001, Jones1993, Finkel2003}) version 4.0 (DIviding RECTangles) is a deterministic\footnote{Here, we refer to the predictable behaviour of the solver given a fixed set of parameters.} global optimization solver first introduced in \cite{Jones1993} as an extension of Lipschitzian optimization. DIRECT does not require information about the gradient nor about the Lipschitz constant and, hence, can be used for black-box functions. DIRECT divides the search domain into rectangles and evaluates the function at the centre of each rectangle. Based on the previously sampled points, DIRECT carefully decides what rectangle to divide next balancing between local and global searches. 
Jones et al. \cite{Jones1993} showed that DIRECT is guaranteed to converge to global minimum, but convergence may sometimes be slow.

BARON(\cite{Sahinidis2014, Sahinidis2005}) version 17.10.10 (Branch-And-Reduce Optimization Navigator) is a branch- and-bound type global optimization solver for non-linear and mixed-integer non-linear programs. To provide lower and upper bounds for each branch, BARON utilizes algebraic structure of the objective function. It also includes a preprocessing step where it performs a multi-start local search to obtain a tight global upper bound. In comparison to other existing global solvers, BARON was demonstrated to be the most robust and fastest (see \cite{Neumaier2005}). However, BARON accepts only a few (general) classes of functions\footnote{For instance, BARON cannot be applied to problems which include trigonometric functions.} including polynomial, exponential, logarithmic, etc. and, unlike DIRECT, it is unable to optimize black-box functions. 

KNITRO(\cite{Byrd2006}) version 10.3.0 is a large-scale non-linear local optimization solver capable of handling problems with hundreds of thousands of variables. KNITRO allows to solve problems using one of the four algorithms: two interior point type methods (direct and conjugate gradient) and two active set type methods (active set and sequential quadratic programming). In contrast to BARON and DIRECT, which specialize on finding global minima, KNITRO focuses on finding local solutions. Nonetheless, KNITRO has multi-start capabilities, i.e., it solves a problem locally multiple times every time starting from a different point in the feasible domain. It is this feature that we make use of in the experiments.

\paragraph{Generating the test set.}
Our test set of functions with low effective dimensionality will be derived from 19 global optimization problems (of dimensions 2--6) with known global minima \cite{AMPGO, Ernesto2005, Bingham2013}, some of which are from the Dixon-Szego set \cite{Dixon-Szego1975}. The list of the problems is given in \Cref{table: Test set}, \Cref{app: Test set}. 

Below we describe the method adopted from Wang et al. \cite{Wang2016} to generate high-dimensional functions with low effective dimensionality. Let $\bar{g}(\bar{\mvec{x}})$ be any function from \Cref{table: Test set}; let $d_e$ be its dimension and let the given domain be scaled to $[-1, 1]^{d_e}$. We create a $D$-dimensional function $g(\mvec{x})$ by adding $D-d_e$ fake dimensions to $\bar{g}(\bar{\mvec{x}})$, $ g(\mvec{x}) = \bar{g}(\bar{\mvec{x}}) + 0\cdot x_{d_e+1} + 0 \cdot x_{d_e+2} + \cdots + 0\cdot x_{D}$. We further rotate the function by applying a random orthogonal matrix $\mtx{Q}$ to $\mvec{x}$ to obtain a non-trivial constant subspace. The final form of the function we test is given as 
\begin{equation}\label{eq: f=g(Qx)}
	f(\mvec{x}) = g(\mtx{Q}\mvec{x}).
\end{equation}
Note that the first $d_e$ rows of $\mtx{Q}$ now span the effective subspace $\mathcal{T}$ of $f(\mvec{x})$. Furthermore, 
\begin{equation}\label{ineq: upper_bound_on_mu}
	\mu:= \min_{\mvec{x} \in \mathcal{G}_1} f(\mvec{x}) = \min_{\bar{\mvec{x}} \in \mathcal{G}_2} \bar{g}(\bar{\mvec{x}}) \leq \sqrt{d_e},
\end{equation}
where $\mathcal{G}_1$ and $\mathcal{G}_2$ are the sets of global minimizers of $f$ and $\bar{g}$, respectively.  

For each problem in the test set, we generate three functions $f$ as defined in \eqref{eq: f=g(Qx)} one for each $D = 10$, $100$, $1000$. We will tackle \eqref{eq: GO} for each $f$ both directly (we call it `\textit{no embedding}') and applying REGO outlined in \Cref{alg: REGO}. 

\paragraph{Experimental setup (REGO).}
We compare `no embedding' and REGO using the three solvers above. Let $g_i$, $s_j$, $n_j$ and $D_k$ denote the $i$th function in the problem set ($g_1$ = Beale, etc., see \Cref{table: Test set}), $j$th solver ($s_1$ = DIRECT, $s_2$ = BARON, $s_3$ = KNITRO), the total number of problems in the problem set solvable by $j$th solver ($n_1$ = 19, $n_2 = 15$, $n_3 = 18$) and $k$th ambient dimension ($D_1 = 10$, $D_2 = 100$, $D_3 = 1000$), respectively. Let $f_{ik}$ denote the $D_k$-dimensional function with low effective dimensionality constructed from $g_i$ as described previously. 

Within `no embedding' framework, for each pair ($s_j$,$D_k$), we solve $f_{ik}$ for  $i = 1, 2,\dots, n_j$ with solver $s_j$ and record the proportion of the problems that attain convergence (see definition in \Cref{table: solvers_descriptions}). 

For each $f_{ik}$ ($1\leq i \leq n_j$, $1 \leq k \leq 3$), we apply REGO 100 times every time with a different Gaussian matrix. Thus, in total, for each pair ($s_j$,$D_k$) we solve $n_j \times 100$ problems. We record the proportion of problems that attain convergence (see \Cref{table: solvers_descriptions}) out of these $n_j \times 100$ problems.

We also record the number of function evaluations (for DIRECT and KNITRO) and CPU time (for all the three solvers) spent before termination within the two frameworks. For each ($s_j$,$D_k$), function evaluations and time are averaged out over $n_j \times 100$ problems within REGO and over $n_j$ problems within `no embedding'.

We conduct the above experiment for REGO with the following pairs of parameters ($d$,$\delta$): $( d_e, 8.0\times\sqrt{d_e} ), (d_e+1, 2.2\times\sqrt{d_e}), (d_e+2, 1.3 \times \sqrt{d_e})$ and $(d_e+3, 1.0 \times \sqrt{d_e})$. Here, each $\delta$ was set to $M\bar{\delta}$, where $M = \sqrt{d_e}$ is an upper bound on $\mu$ (see \eqref{ineq: upper_bound_on_mu}) and the value for $\bar{\delta}$ was chosen as the smallest $\bar{\delta}$ that gives at least $90\%$ chance of success based on the curve of $R^*$ in \Cref{fig: par_est}.

\paragraph{Experimental setup (solvers).}
Due to the difference in algorithmic procedures of the solvers, they allow different budget constraints and have different convergence and termination criteria; we present these in \Cref{table: solvers_descriptions}. 
\begin{table}[!t]
	\centering
	\caption{ The table outlines the experimental setup for the three solvers. In the table, $f$ is a function with low effective dimensionality $d_e$ and the global minimum $f^*$, and $\epsilon$ is set to $10^{-3}$.}
	\label{table: solvers_descriptions}
	\begin{tabular}{p{2.42cm}|p{3.9cm}|p{3.9cm}|p{3.9cm}}
		& \multicolumn{1}{c | }{DIRECT} & \multicolumn{1}{c |}{BARON} & \multicolumn{1}{c}{KNITRO}\\ \hline
		\mbox{Measure of} computational cost& function evaluations & CPU seconds & function evaluations, CPU seconds \\ \hline
		 \mbox{Budget per} problem & $10000\times d_e$ function evaluations & $200\times d_e$ CPU seconds & $20 \times d_e$ starting points \\ \hline
		 Convergence criteria (see \Cref{rem: convergence criteria})& $f^*_D \leq f^* + \epsilon$ &
		 Convergence: \mbox{$f_B^U \leq f^* + \epsilon$} Convergence$_{opt}$: $f_B^U \leq f_B^L + \epsilon$
		  & $f_K^* \leq f^* + \epsilon$
		  \\ \hline
		 Termination criteria & Either on budget or if $\mvec{x}_D^*$ satisfies the convergence criteria &  Either on budget or if $f_B^U$ and $f_B^L$ satisfy the convergence$_{opt}$ criteria & On budget \\ \hline
		 Additional \hspace{0.2cm} options & \verb|options.testflag|=1 \verb|options.maxits|=\verb|Inf| \verb|options.globalmin|=$f^*$ & \mbox{\texttt{npsol} = 9} \hspace{0.3cm}
		\mbox{\texttt{numloc} = 0} \mbox{\texttt{BrVarStra} = 1} \mbox{\texttt{BrPtStra} = 1}  & Default options. Derivatives are allowed. Use of multi-start through \verb|ms_enable|=1.  \\ \hline
	\end{tabular}
\end{table}

\begin{rem}\label{rem: convergence criteria}
	DIRECT, at its every iteration, stores $f_D^*$ --- the minimum value of $f$ so far found. BARON, at its every iteration, stores $f_B^U$ and $f_B^L$ --- smallest upper bound and largest lower bound so far found for $f$. As for KNITRO, $f_K^*=\min \{ f(\mtx{A}\mvec{y}_1^*), f(\mtx{A}\mvec{y}_2^*), \dots, f(\mtx{A}\mvec{y}_l^*)\}$, where $l$ is the number of starting points and where $\{\mvec{y}_i^*\}_{1 \leq i \leq l}$ are the local solutions produced by the multi-start procedure. 
\end{rem}
\begin{rem}
	The experiments are done not to compare solvers but to contrast `no embedding' with REGO. All the experiments were run in MATLAB on the 16 cores (2$\times$8 Intel with hyper-threading) Linux machines with 256GB RAM and 3300 MHz speed. 
\end{rem}

\subsection{Numerical results}
\paragraph{\eqref{eq: REGO} successful.}
We record the proportion of instances for which \eqref{eq: REGO} is successful. \Cref{table: average_minimum_in_box} presents these percentages for each particular choice of $d$ and $D$ averaged over 19 problems in the test set. 
We observe that the percentages are very high and appear to be independent of $D$ supporting the conclusions of \Cref{cor: REGO_successful_lower_bound}.
\begin{table}[!ht]
	\centering
	\caption{The table shows average percentages of problems for which \eqref{eq: REGO} is successful.}
	\label{table: average_minimum_in_box}
	\begin{tabular}{c|c|c|c}
		$d$/$D$ & 10 & 100 & 1000	\\ \hline
		$d_e + 0$ & 97.2 & 97.8 & 97.3 \\
		$d_e+1$ & 99.1 & 98.9 & 99.3\\
		$d_e+2$ & 99.5 & 99.6 & 99.8 \\
		$d_e+3$ & 100 & 99.9 & 99.8
	\end{tabular}
\end{table}
\paragraph{REGO vs. no embedding.}
The results of the experiment comparing REGO and `no embedding' are presented in \Cref{fig: REGO_DIRECT}, \Cref{fig: REGO_BARON} and \Cref{fig: REGO_KNITRO} for DIRECT, BARON and KNITRO, respectively. These figures compare average proportions of converged solutions and computational costs produced by REGO and `no embedding' frameworks for $D=10,100,1000$. 

DIRECT (\Cref{fig: REGO_DIRECT}). For all the four initialisations of REGO, we observe that the average proportions of problems that attained convergence (see definition in \Cref{table: solvers_descriptions}) are invariant with respect to the ambient dimension. This frequency of convergence is higher within `no embedding' for $D = 10, 100$, but exhibits a significant drop for $D = 1000$. The average function evaluation count is maintained within REGO, but doubles within `no embedding' for a tenfold increase in $D$. Growth in CPU time takes place within both frameworks, being highest for `no embedding'. 

BARON  (\Cref{fig: REGO_BARON}). In comparison with `no embedding', the frequency of convergence$_{opt}$ is higher within REGO in most cases. We note that BARON's both convergence and convergence$_{opt}$ exhibit invariance with respect to $D$ within REGO.  As for `no embedding', we observe a decrease in the frequencies of both convergence and convergence$_{opt}$. In addition, we observe an increase in CPU time spent within `no embedding', whilst the time is almost constant within REGO.

KNITRO (\Cref{fig: REGO_KNITRO}). We see that the proportion of solved problems is invariant with respect to the ambient dimension within REGO and, surprisingly, within `no embedding' as well. However, the average number of function evaluations and time spent differ significantly between the two frameworks. With REGO, the average number of function evaluations remain at the same level for all $D$. Average time grows within both frameworks, but at a higher rate for `no embedding'. The average time differs by a factor of $70$ for $D = 1000$ in favour of REGO. We think that the growth in time within REGO is due to more costly function and derivative evaluations for larger $D$. 

\begin{figure}[!t]
	\centering
	\includegraphics[scale = 0.6]{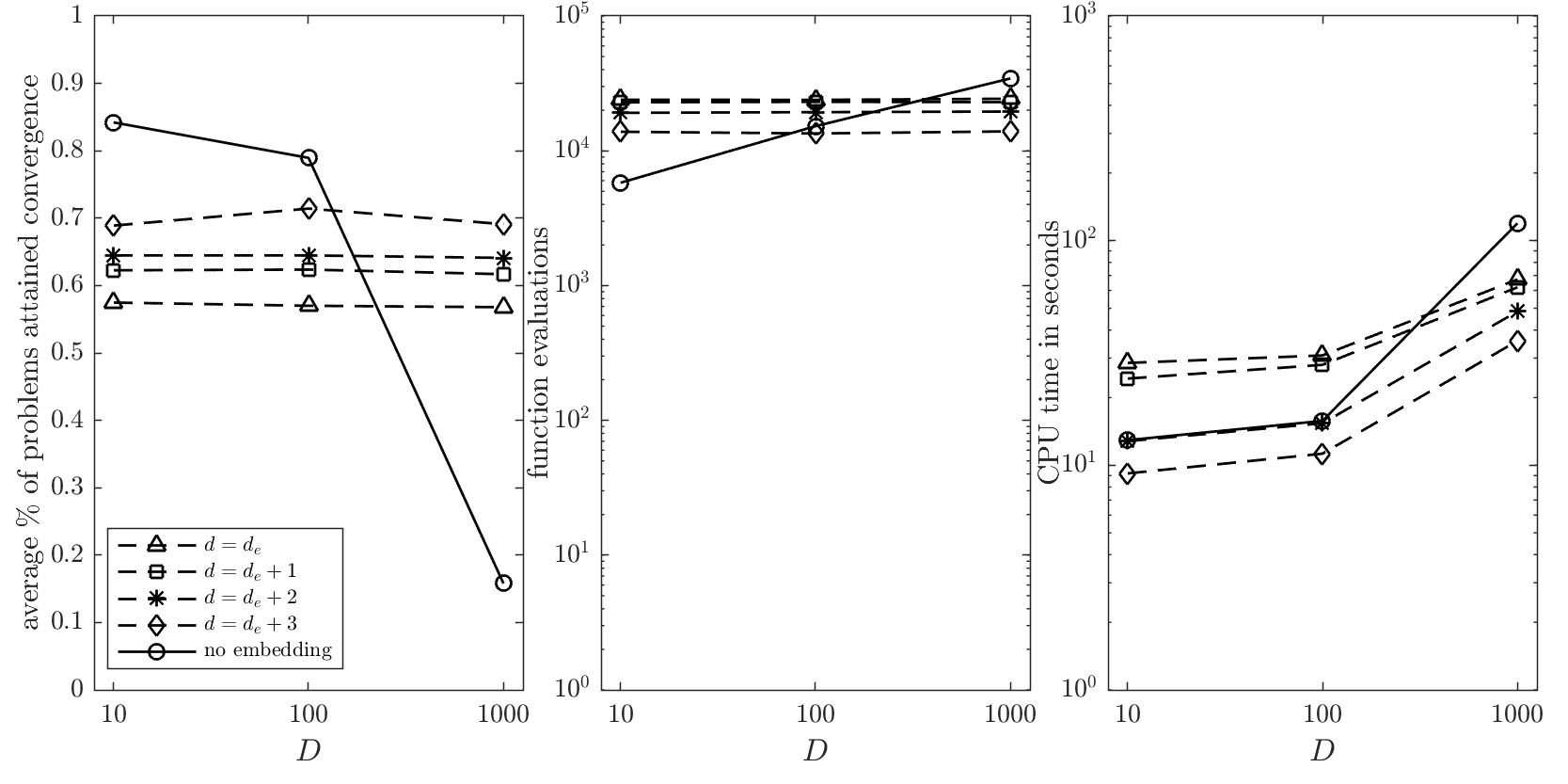}
	\caption{REGO versus `no embedding' with DIRECT: comparison of frequency of convergence, log of average function evaluations and log of average CPU time (in seconds).}
	\label{fig: REGO_DIRECT}
\end{figure}
\begin{figure}[!ht]
	\centering
	\includegraphics[scale = 0.6]{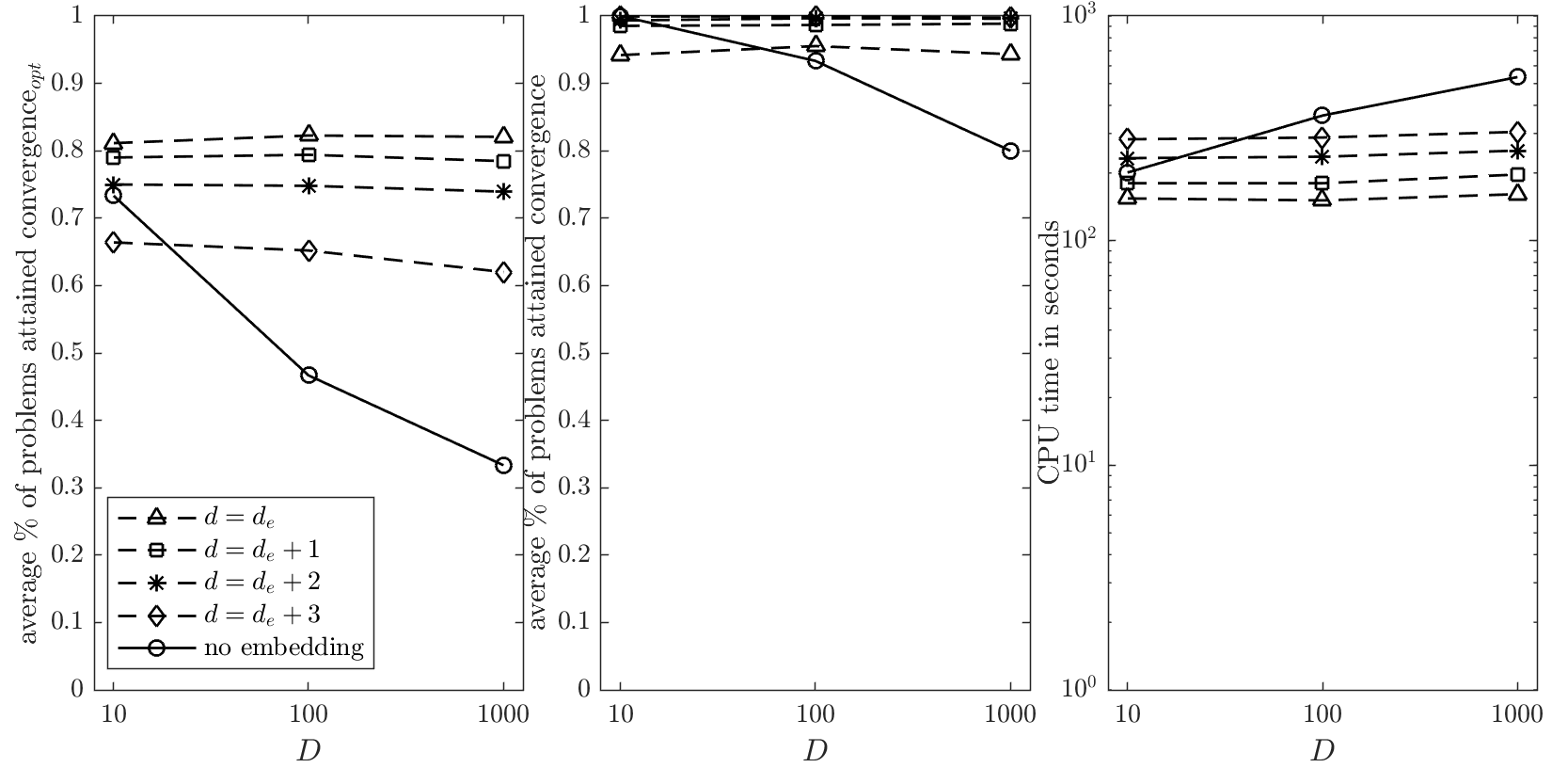}
	\caption{REGO versus `no embedding' with BARON: comparison of frequency of convergence$_{opt}$/convergence and average CPU time (in seconds). }
	\label{fig: REGO_BARON}
\end{figure}
\begin{figure}[!ht]
	\centering
	\includegraphics[scale = 0.6]{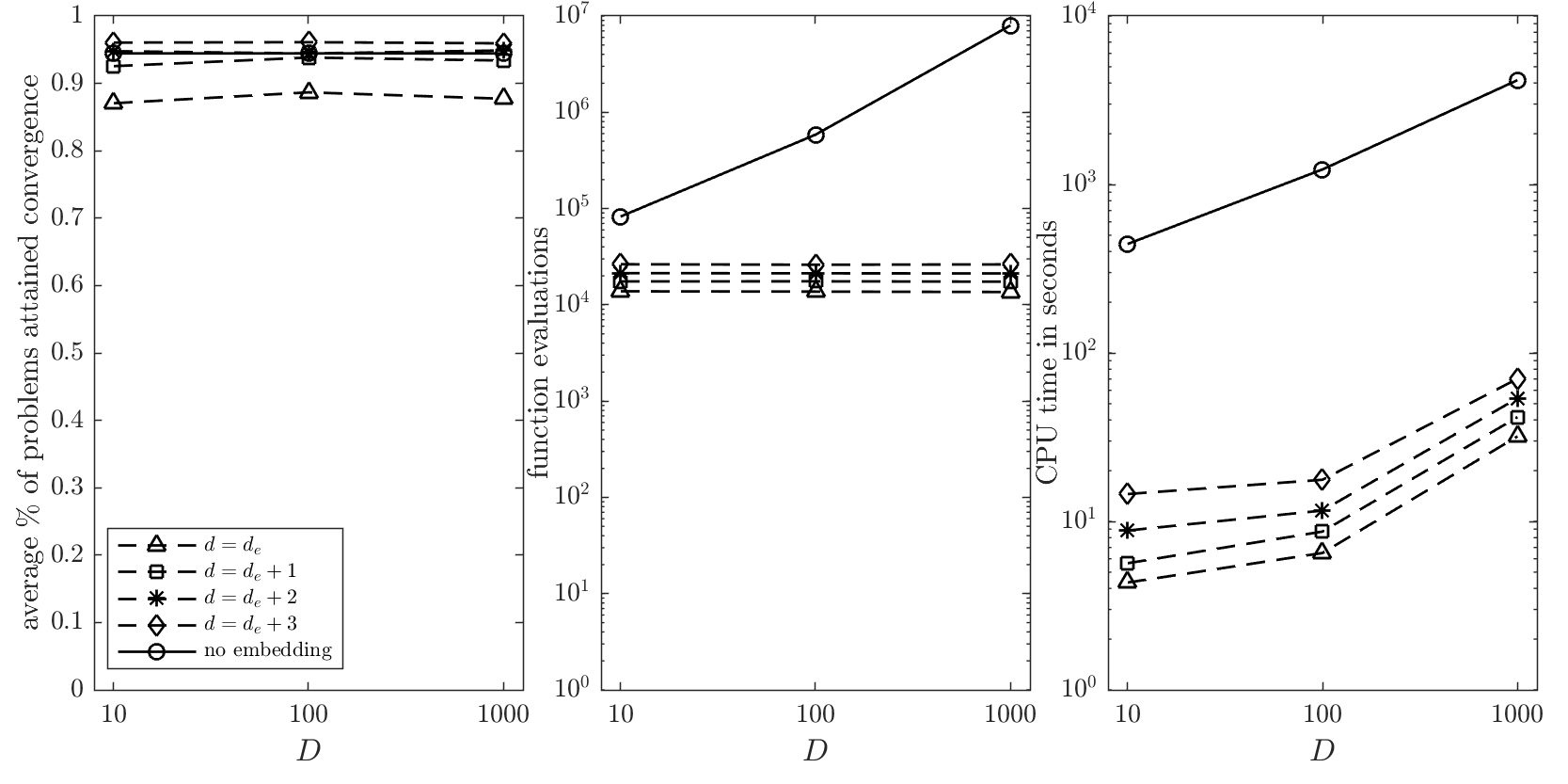}
	\caption{REGO versus `no embedding' with KNITRO: comparison of frequency of convergence, log of average function evaluations and log of average CPU time (in seconds).}
	\label{fig: REGO_KNITRO}
\end{figure}

\subsection*{Summary of numerical results}
\begin{enumerate}
	\item (Effects of parameter choices) Our experiments clearly show that the choice of $d$ and $\delta$ has a considerable effect on convergence and computational cost of REGO, and that good choices of $(d,\delta)$ are dependent on the underlying solver. For example, BARON achieved highest proportion of convergence$_{opt}$ within least amount of time for $(d,\delta) = (d_e, 8\sqrt{d_e})$, whereas DIRECT performed best for $(d_e +3, \sqrt{d_e})$. KNITRO produced highest proportion of convergence and worst time for $(d_e +3, \sqrt{d_e})$, and lowest proportion of convergence and best time for $(d_e, 8\sqrt{d_e})$.
	
	\item (Scalability) Within REGO, the proportion of problems solved and/or number of function evaluations are generally invariant with respect to the ambient dimension $D$. REGO displays good scalability for all three solvers. 
	
	\item (No embedding) Within `no embedding', as $D$ increases, the proportion of problems that attained convergence$_{opt}$/convergence decreased for BARON and DIRECT. Surprisingly, for the KNITRO's multi-start method, the proportion of solved problems is maintained, but the number of function evaluations and time increased dramatically.  
\end{enumerate}

\paragraph{Additional experiments.} To see how robust REGO is to the changes in the parameters, we conduct three more experiments presented and discussed in \Cref{app: more_experiments}. In the first experiment, assuming that $\mu$ is bounded by $\sqrt{D}$ (see page \pageref{par: choosing_d_and_delta_in_practice} for an explanation for this choice), we set $\delta$ to $\sqrt{D} \bar{\delta}$ for $\bar{\delta}$ chosen as in the main experiment. The second experiment tests REGO for four different values of $d$ while keeping $\delta$ fixed and the third experiment tests REGO for three different values of $\delta$ keeping $d$ fixed. In all three experiments, REGO performs well, particularly for BARON and KNITRO, solving most of the problems and exhibiting similar trends as in the main experiment. 


\section{Conclusions and future work} \label{sec: conclusions}

We study a general algorithmic framework for functions with low effective dimensionality that solves the reduced problem \eqref{eq: REGO} using a single Gaussian random embedding and a(ny) general global optimization solver.
Our precise theoretical findings backed by the numerical experiments show that the success of \eqref{eq: REGO} is essentially independent of $D$ and mainly depends on the gap between the embedding dimension $d$ and the dimension of effective subspace $d_e$, and the ratio between the size of $\mathcal{Y}$ (namely $\delta$) and $\mu$ (the Euclidean distance to the closest affine subspace of minimizers). REGO with three standard global solvers produced high frequencies of convergence, generally outperforming  the respective solver's performance when applied directly to the problems (without the dimensionality reduction)  in terms of proportion of problems solved and/or computational cost. 

Our in-depth investigations are conceptual in nature, and there is clearly more work that needs to be done to make this framework practically applicable to global optimization problems with special structure. In particular, as outlined on
page \pageref{par: choosing_d_and_delta_in_practice}, our REGO approach depends on knowing (an upper bound $d$ on)
the effective dimension $d_e$. Future work may include estimating $d_e$ prior to optimizing, noting that REGO does not need to learn the entire effective subspace only its dimension. One could also estimate $d$ or $d_e$ numerically, as proposed in \cite{Sanyang2016}, where $d$ is gradually increased until no significant changes in the best function value found are observed. Our theoretical choices for $\delta$ also depend on $\mu$, which again needs estimating. In this case, choosing a box domain for $f$ would provide a rough estimate for $\mu$ (as discussed on page  \pageref{par: choosing_d_and_delta_in_practice}), with the remark that REGO cannot (yet) guarantee feasibility with respect to given bounds. To achieve the latter, one needs to either add projection operators as in \cite{Wang2016, Binois2014, Binois2017}, or include the problem constraints in the formulation of \eqref{eq: REGO} and allow multiple random embeddings as in \cite{QianHuYu2016}. Alternative potential directions include investigating other random or deterministic matrix choices for the embeddings, as  considered for example in \cite{Nayebi2019}. Real-life problems are often only approximately low dimensional, and so their optimization requires further extensions and analysis of the random embedding framework. 


\bibliographystyle{plainnat}
{\footnotesize
	\bibliography{bibliography}}

\appendix
\section{Technical definitions and results}
\subsection{Gaussian random matrices}

\begin{definition}[Gaussian matrix] \label{def: Gaussian_matrix}
	A Gaussian (random) matrix is a matrix whose each entry is an independent standard normal random variable. 
\end{definition}
Gaussian matrices have been well-studied with many results available at hand. Here, we mention a few key properties of Gaussian matrices that we use in the analysis; for a collection of results pertaining to Gaussian matrices and other related distributions refer to \cite{Gupta1999, Vershynin2018}. 

Gaussian random matrices are known to be invariant with respect to orthogonal transformations:
\begin{theorem} \label{thm: orthog_inv_of_Gaussian_matrices} (see \cite[Theorem 2.3.10]{Gupta1999})
	Let $\mtx{A}$ be an $D \times d$ Gaussian random matrix. If $\mtx{U} \in \mathbb{R}^{D \times p}$, $D \geq p $, and $\mtx{V} \in \mathbb{R}^{d \times q}$, $d \geq q$, are orthogonal, then $\mtx{U}^T\mtx{A}\mtx{V}$ is a Gaussian random matrix. 
\end{theorem}

A related notion that plays an important role in the study of Gaussian matrices is the Wishart distribution represented by matrix $\mtx{A}^T\mtx{A}$ (or $\mtx{A}\mtx{A}^T$), where $\mtx{A}$ is an overdetermined (underdetermined) Gaussian matrix. A Wishart matrix is positive definite with probability 1:
\begin{theorem} \label{thm: Wishart_is_nonsingular} (see \cite[Theorem 3.2.1]{Gupta1999})
	Let $\mtx{A}$ be an $D \times d$ Gaussian random matrix, $D \geq d$. Then, the Wishart matrix $\mtx{A}^T\mtx{A}$ is positive definite with probability 1. 
\end{theorem}
The immediate consequence of the result is that the Wishart matrix is nonsingular with probability 1.

\subsection{Chi-squared random variable} \label{def: chi-squared_rv}
\begin{definition}[Chi-squared random variable] 
	Given a collection $Z_1, Z_2, \dots, Z_n$ of $n$ independent standard normal variables, a random variable $X = Z_1^2+Z_2^2+\cdots Z_n^2$ is said to follow the chi-squared distribution with $n$ degrees of freedom. We denote this by $X \sim \chi^2_n$. 
\end{definition}
The following lemma provides a notable relationship between the inverse of the Wishart matrix and the chi-squared random variable.
\begin{lemma}\label{lemma: Rayleigh_quotient_chi_square} (see \cite[Corollary 3.3.13.1.]{Gupta1999})
	Let $\mtx{A}$ be an $D \times d$ Gaussian matrix, $D \geq d$, and let $\mvec{z} \in \mathbb{R}^{d}$ be a fixed non-zero vector. Then,
	$$ \frac{\| \mvec{z} \|^2}{\mvec{z}^T (\mtx{A}^T\mtx{A})^{-1} \mvec{z} } \sim \chi^2_{D-d+1}.$$
\end{lemma}
In the following lemma we derive an upper bound for the cumulative density function (c.d.f.) of the chi-squared random variable.
\begin{lemma}\label{lemma: chi-sq_tail_bound}
	Let $X \sim \chi^2_n$. Then,
	$$ \prob[X \leq x] \leq \frac{4}{n(n+2)\Gamma(n/2)}\bigg(1+\frac{n}{2}e^{-x/2}\bigg)(x/2)^{n/2}. $$
\end{lemma}
\begin{proof}
	Recall the c.d.f. of the chi-square random variable (see, e.g., \cite{NIST_chi}):
	$$ \prob[X \leq x] = \frac{\gamma(n/2,x/2)}{\Gamma(n/2)}$$
	for $x > 0$, where $\gamma(n/2,x/2)$ is the lower incomplete gamma function (see \cite{Neuman2013}) defined as 
	$$ \gamma(n/2,x/2) = \int_{0}^{x/2} u^{n/2-1}e^{-u}du. $$
	We obtain the desired result by applying the following upper bound on $\gamma(n/2,x/2)$ (see \cite[Theorem 4.1]{Neuman2013}),
	$$\gamma(n/2,x/2) \leq \frac{4}{n(n+2)}\bigg(1+\frac{n}{2}e^{-x/2}\bigg)(x/2)^{n/2}.$$
\end{proof}

\subsection{The inverse chi-squared random variable}
\begin{definition}[Inverse chi-squared random variable] \label{def: inv_chi-squared}
	Given $X \sim \chi_n^2$, a random variable $Y = 1/X$ is said to follow the inverse chi-squared distribution with $n$ degrees of freedom. We denote this by $Y \sim 1/\chi_n^2$ (see \cite[A5]{Lee2012}).  
\end{definition}
\begin{lemma}\label{lemma: expectation_inverse_chi_squared} (see \cite[A5]{Lee2012})
	Let $Y \sim 1/\chi_n^2$ and $W = sY$ for a positive real $s$. Then,
	$$ \mathbb{E}[W] = \frac{s}{n - 2} $$
	provided that $n > 2$. 
\end{lemma}

\begin{lemma}\label{lemma: pdf_of_sqrt_Inv_Chi}
	Let $Y$ and $R$ be two random variables such that $Y \sim 1/\chi^2_n$ and $R = \sqrt{Y}$. Then, the probability density function (p.d.f.) $g(\hat{r})$ of $R$ is given by
	$$ g(\hat{r}) = \frac{2^{-n/2+1}}{\Gamma(n/2)} \hat{r}^{-n-1} e^{-1/(2\hat{r}^2)}. $$
\end{lemma}
\begin{proof}
	The p.d.f. $g(\hat{r})$ of $R$ satisfies
	\begin{equation} \label{eq: 2tq(t^2)}
	g(\hat{r}) = \frac{d}{d\hat{r}} \prob[ \sqrt{Y} < \hat{r}] = \frac{d}{d\hat{r}} \prob[Y < \hat{r}^2] = 2\hat{r}h(n, \hat{r}^2),
	\end{equation}
	where  $h(n, \cdot)$ denotes the p.d.f. of the inverse chi-squared random variable $Y$ given by (see, e.g., \cite[A5]{Lee2012})
	$$\text{$h(n,y) = \frac{1}{2^{n/2}\Gamma(n/2)}y^{-n/2-1}e^{-1/(2y)}$ for $y > 0$.}$$
\end{proof}

\subsection{Spherically distributed random vectors}

\begin{definition}\label{def: spherical_distribution}
	An $D\times 1$ random vector $\mtx{x}$ is said to have a spherical distribution if for every orthogonal $D \times D$ matrix $\mtx{U}$,
	$$ \mtx{U}\mvec{x} \stackrel{law}{=} \mvec{x}.$$
\end{definition}

Below are some useful facts about symmetrically distributed random vectors. 

\begin{lemma} \label{lemma: Fang_equally_distributed_random_vectors} \cite[p.~13]{Fang1990}
	Let $\mvec{x}$ and $\mvec{y}$ be random vectors such that $\mvec{x} \stackrel{law}{=} \mvec{y}$ and let $f_i(\cdot)$, $i=1,2,\dots,m$, be measurable functions. Then,
	$$ \begin{pmatrix}
	f_1(\mvec{x}) & f_2(\mvec{x}) & \dots & f_m(\mvec{x}) 
	\end{pmatrix}^T \stackrel{law}{=}  \begin{pmatrix}
	f_1(\mvec{y}) & f_2(\mvec{y}) & \dots & f_m(\mvec{y})
	\end{pmatrix}^T.$$
	
\end{lemma}

\begin{lemma}\label{lemma: x=ru} (\cite[Corollary, p. 30]{Fang1990})
	If $D\times 1$ random vector $\mvec{x}$ has a spherical distribution, then
	$$ \mvec{x} \stackrel{law}{=} r\mvec{u}, $$
	where $\mvec{u}$ is distributed uniformly on the unit sphere $S^D$ and $r$ is a univariate random variable independent of $\mvec{u}$.
	
\end{lemma}

\begin{theorem} \label{thm: Fang1990_r=x_and_u=x^-1x} (\cite[Theorem 2.3]{Fang1990})
	Let $\mvec{x} \stackrel{law}{=} r\mvec{u}$ be a spherically distributed $D \times 1$ random vector with $\prob[\mvec{x} = \mvec{0}] = 0$. Then,
	$$ \text{$\| \mvec{x} \| \stackrel{law}{=} r$ and $\| \mvec{x} \|^{-1}\mvec{x} \stackrel{law}{=} \mvec{u}$}, $$
	Moreover, $\| \mvec{x} \|$ and $\|\mvec{x}\|^{-1} \mvec{x}$ are independent. 
\end{theorem}

\begin{theorem} \label{thm: pdf_of_x_Gupta} (see \cite[Theprem 2.1.]{Gupta1997})
	Let $\mvec{x} \stackrel{law}{=} r\mvec{u}$ be a spherically distributed $D \times 1$ random vector with $\prob[\mvec{x} = \mvec{0}] = 0$, where $r$ is independent of $\mvec{u}$ with p.d.f. $h(\cdot)$. Then, p.d.f. $g(\hat{\mvec{x}})$ of $\mvec{x}$ is given by 
	$$ g(\hat{\mvec{x}}) = \frac{\Gamma(n/2)}{2\pi^{D/2}} h(\| \hat{\mvec{x}} \|) \| \hat{\mvec{x}} \|^{1-D}. $$
\end{theorem}

For more details regarding spherical distributions refer to \cite{Fang1990, Gupta1999, Bernardo2000}. 

\subsection{The least Euclidean norm solution to the  random linear system}

The present section establishes key properties of the least Euclidean norm solution to the underdetermined random linear system.

\begin{enumerate}
	\item[(C)] Let $\bar{\mtx{B}}$ be a $d_e \times d$ Gaussian matrix, where $d_e \leq d$, and let $\mvec{z} \in \mathbb{R}^{d_e}$ be a fixed nonzero vector. Denote by $\mvec{y}_2$ the least $2$-norm solution to $\bar{\mtx{B}}\mvec{y} = \mvec{z}$.
\end{enumerate}

\begin{lemma}\label{lemma: inverse_min_norm_solution_follow_chi-square}
	Given (C), $\mvec{y}_2$ satisfies 
	$$ \frac{\| \mvec{z} \|^2_2}{\| \mvec{y}_2 \|^2_2} \sim \chi^2_{d-d_e+1}.$$
\end{lemma}
\begin{proof}
	The least Euclidean norm solution $\mvec{y}_2$ to $\bar{\mtx{B}}\mvec{y} = \mvec{z}$ is given by
	$$ \mvec{y}_2 = \bar{\mtx{B}}^T(\bar{\mtx{B}}\bar{\mtx{B}}^T)^{-1}\mvec{z}. $$
	For its Euclidean norm, we have
	\begin{equation*} \label{eq: norm_of_y_derivation}
	\begin{aligned}
	\| \mvec{y}_2 \|^2_2 & = (\bar{\mtx{B}}^T(\bar{\mtx{B}}\bar{\mtx{B}}^T)^{-1}\mvec{z})^T\bar{\mtx{B}}^T(\bar{\mtx{B}}\bar{\mtx{B}}^T)^{-1}\mvec{z} \\
	& = \mvec{z}^T (\bar{\mtx{B}}\bar{\mtx{B}}^T)^{-1} \mvec{z},
	\end{aligned}
	\end{equation*}
	Using \Cref{lemma: Rayleigh_quotient_chi_square} we obtain the desired result:
	$$ \frac{\|\mvec{z}\|^2_2}{\|\mvec{y}_2\|^2_2} = \frac{\| \mvec{z} \|^2}{\mvec{z}^T (\bar{\mtx{B}}\bar{\mtx{B}}^T)^{-1} \mvec{z}} \sim \chi^2_{d-d_e+1}. $$
\end{proof}

\begin{lemma} \label{lemma: y_has_spherical_distr}
	Given (C), $\mvec{y}_2$ follows a spherical distribution. 
\end{lemma}
\begin{proof}
	Let $\mtx{S}$ be any $d\times d$ orthogonal matrix. Let $f : \mathbb{R}^{d_e d \times 1} \rightarrow \mathbb{R}^{d\times 1}$ be a vector-valued function defined as
	$$ f(\vc(\bar{\mtx{B}})) = \bar{\mtx{B}}^T(\bar{\mtx{B}}\bar{\mtx{B}}^T)^{-1}\mvec{z},$$
	where $\vc(\bar{\mtx{B}})$ denotes the $Dd\times 1$ vector $(\bar{\mvec{b}}_1^T \; \bar{\mvec{b}}_2^T \; \cdots \bar{\mvec{b}}_d^T)^T$ with $\bar{\mvec{b}}_i$ being the $i$th column vector of $\bar{\mtx{B}}$. Using the fact that the inverse of a matrix is equal to the ratio of its adjugate to its determinant we can express $f$ as
	$$ f(\vc(\bar{\mtx{B}})) = \begin{pmatrix}
	\frac{p_1(\bar{\mtx{B}})}{q(\bar{\mtx{B}})} & \frac{p_2(\bar{\mtx{B}})}{q(\bar{\mtx{B}})} & \dots & \frac{p_d(\bar{\mtx{B}})}{q(\bar{\mtx{B}})} 
	\end{pmatrix}^T, $$
	where $p_i(\bar{\mtx{B}})$ for $1 \leq i \leq d$ are some polynomials of the entries of $\bar{\mtx{B}}$ and $q(\bar{\mtx{B}})$ is the determinant of $\bar{\mtx{B}}\bar{\mtx{B}}^T$. 
	
	We first would like to prove that $f$ is a measurable function. Recall that a function is measurable if and only if each of its components is measurable. It is enough to show that $p_1/q$ is measurable; the same argument will apply to the rest of its components. First, we note that
	\begin{enumerate}[label={(\roman*)}]
		\item $p_1$ and $q$ are measurable;
		\item $q$ is non-zero almost everywhere.
	\end{enumerate}
	To prove (i), observe that the polynomials $p_1$ and $q$ are sums of scalar multiples of products of standard normal random variables, which by definition are measurable. Sums, scalar multiples and products of measurable functions are measurable; hence, $p_1$ and $q$ must be measurable. To prove (ii), we refer to \Cref{thm: Wishart_is_nonsingular}, which says that the matrix $\bar{\mtx{B}}\bar{\mtx{B}}^T$ is positive definite with probability 1 implying that all of its eigenvalues are strictly positive with probability 1. Then, (ii) follows from the fact that the determinant of the symmetric square matrix is equal to the product of its eigenvalues. Now, we can apply \cite[Theorem 4.10]{Wheeden2015} to deduce that $p_1/q$ is measurable; this completes the proof that $f$ is measurable. 
	
	For $\mvec{y}_2 = \bar{\mtx{B}}^T(\bar{\mtx{B}}\bar{\mtx{B}}^T)^{-1}\mvec{z}$, we have
	$$\text{$\mvec{y}_2 =  f(\vc(\bar{\mtx{B}}))$ and $\mtx{S}\mvec{y}_2 = f(\vc(\bar{\mtx{B}}\mtx{S}^T))$}.$$
	According to \Cref{thm: orthog_inv_of_Gaussian_matrices}, $\vc(\bar{\mtx{B}}) \stackrel{law}{=} \vc(\bar{\mtx{B}}\mtx{S}^T)$. Then, by applying \Cref{lemma: Fang_equally_distributed_random_vectors}, we obtain 
	$$ \mvec{y}_2 \stackrel{law}{=} \mtx{S}\mvec{y}_2. $$
	Hence, $\mvec{y}_2$ follows a spherical distribution by \Cref{def: spherical_distribution}.
\end{proof}

\begin{lemma}\label{lemma: pdf_of_y_general_case}
	Given (C), the probability density function of $\mvec{y}_2$ is given by
	$$ g (\hat{\mvec{y}}) = \pi^{-d/2} \bigg[\frac{\Gamma(d/2)}{\Gamma(n/2)}\bigg] \bigg(\frac{\| \mvec{z} \|}{\sqrt{2}}\bigg)^n   (\hat{\mvec{y}}^T\hat{\mvec{y}})^{-(n+d)/2} e^{-\|\mvec{z}\|^2/(2\hat{\mvec{y}}^T\hat{\mvec{y}})},$$
	where $n = d-d_e+1$.
\end{lemma}
\begin{proof}
	The fact that $\mvec{y}_2$ has a spherical distribution is a key ingredient in the proof. To simplify the derivations, let us assume for now that $\| \mvec{z} \| = 1$. 
	
	By \Cref{lemma: x=ru}, 
	$$ \mvec{y}_2 \stackrel{law}{=} r\mvec{u},  $$
	where $r$ is a univariate random variable and where $\mvec{u}$ is a random vector distributed uniformly on $S^d$; moreover, $r$ and $\mvec{u}$ are independent. 
	
	Our first goal is to show that $r$ and $\| \mvec{y}_2 \|$ have the same distribution. This fact follows immediately from \Cref{thm: Fang1990_r=x_and_u=x^-1x} if we show that $\prob[\mvec{y}_2 = \mvec{0}] = 0$. Let $W \sim 1/\chi^2_{d-d_e+1}$. According to \Cref{lemma: inverse_min_norm_solution_follow_chi-square}, we have 
	\begin{equation}\label{eq: y_distributed_as_W}
	\| \mvec{y}_2 \|^2 \stackrel{law}{=} W,
	\end{equation}
	which we use in the second equation below
	$$ \prob[\mvec{y}_2 = \mvec{0}] = \prob[\| \mvec{y}_2 \|^2 = 0] = \prob[W = 0]. $$
	Since $W$ is a continuous random variable, $ \prob[W = 0] = 0$. This proves that
	\begin{equation}\label{eq: r_distributed_as_y}
	r \stackrel{law}{=} \|\mvec{y}_2 \|.
	\end{equation}
	
	Combining \eqref{eq: y_distributed_as_W} and \eqref{eq: r_distributed_as_y}, we conclude that $r$ has the same distribution as $W^{1/2}$. The probability density function of $W^{1/2}$ --- and, consequently, of $r$ --- derived in \Cref{lemma: pdf_of_sqrt_Inv_Chi} is given by
	\begin{equation} \label{eq: pdf_of_r}
	h(\hat{r}) = \frac{2^{-n/2+1}}{\Gamma(n/2)} \hat{r}^{-n-1} e^{-1/(2\hat{r}^2)}.
	\end{equation}
	\Cref{thm: pdf_of_x_Gupta} allows us to express the p.d.f. of $\mvec{y}_2$ in terms of the p.d.f. of $r$:
	$$ g(\hat{\mvec{y}}) = \frac{\Gamma(d/2)}{2\pi^{d/2}} (\hat{\mvec{y}}^T\hat{\mvec{y}})^{(1-d)/2} h\big(\sqrt{\hat{\mvec{y}}^T \hat{\mvec{y}}}\big). $$
	By using \eqref{eq: pdf_of_r} for $h(\cdot)$ in the above, we obtain
	\begin{equation}\label{eq: g_pdf_of_y_z=1}
	g (\hat{\mvec{y}}) = 2^{-n/2}\pi^{-d/2} \frac{\Gamma(d/2)}{\Gamma(n/2)} (\hat{\mvec{y}}^T\hat{\mvec{y}})^{-(n+d)/2} e^{-1/(2\hat{\mvec{y}}^T\hat{\mvec{y}})}.
	\end{equation}
	
	To derive the p.d.f. for arbitrary non-zero $\mvec{z}$, we consider the linear transformation $\bar{\mvec{y}} = \| \mvec{z} \| \hat{\mvec{y}}$. The Jacobian of the transformation is equal to $1/\|\mvec{z}\|^d$. Thus, the p.d.f. $\bar{g}(\bar{\mvec{y}})$ of $\bar{\mvec{y}}$ satisfies
	$$ \bar{g}(\bar{\mvec{y}}) = \frac{g(\bar{\mvec{y}}/\| \mvec{z} \|)}{\| \mvec{z} \|^d}, $$
	which together with \eqref{eq: g_pdf_of_y_z=1} yields the desired result.
\end{proof}

\section{Problem set}  \label{app: Test set}
\Cref{table: Test set} contains the explicit formula, domain and global minimum of the functions used to generate the high-dimensional test set. The problem set contains 19 problems taken from \cite{AMPGO, Ernesto2005, Bingham2013}. Problems that cannot be solved by BARON are marked with `$^*$'. Problems that will not be solved by KNITRO are marked with `$^\circ$'.
\begin{table}[!ht]
	\centering
	\caption{The problem set listed in alphabetical order.}
	\label{table: Test set}
	\begin{tabular} {|L{4.1cm} | C{3cm} | C{3.3cm} | }
		\hline
		Function & Domain  & Global minima  \\ \hline
		1) Beale \cite{Ernesto2005}    & $\mvec{x} \in [-4.5,4.5]^2$ &  $g(\mvec{x}^*) = 0$ \\ \hline
		2) Branin$^*$ \cite{Ernesto2005}    & \pbox{20cm}{
			$x_1 \in [-5,10]$ \\ $x_2 \in [0, 15]$} 
		&
		$g(\mvec{x}^*) = 0.397887$ \\ \hline
		
		3) Brent \cite{AMPGO}  & $\mvec{x} \in [-10,10]^2$ & $g(\mvec{x}^*) = 0$  \\ \hline
		
		4) $^\circ$Bukin N.6   \cite{Bingham2013}   & \pbox{20cm}{ $ x_1 \in [-15,-5]$ \\ $x_2 \in [-3,3]$} & $g(\mvec{x}^*) = 0$ \\ \hline
		
		5) $^*$Easom  \cite{Ernesto2005} & $\mvec{x}\in [-100,100]^2$   & $g(\mvec{x}^*) = -1$ \\ \hline
		
		6) Goldstein-Price \cite{Ernesto2005} & $\mvec{x} \in [-2,2]^2 $  & $g(\mvec{x}^*) = 3$ \\ \hline
		
		7) Hartmann 3 \cite{Ernesto2005} & $\mvec{x} \in [0,1]^3$ & $g(\mvec{x}^*) = -3.86278$
		\\ \hline
		
		8) Hartmann 6 \cite{Ernesto2005}  & $\mvec{x} \in [0,1]^6$  & $g(\mvec{x}^*) = -3.32237$
		\\ \hline
		
		9) $^*$Levy \cite{Bingham2013}  & $\mvec{x} \in [-10,10]^4$ & $g(\mvec{x}^*) = 0$  \\ \hline
		
		10) Perm 4, 0.5 \cite{Bingham2013} & $\mvec{x} \in [-4,4]^4$ & $g(\mvec{x}^*) = 0$  \\ \hline
		
		11) Rosenbrock \cite{Bingham2013}   & $\mvec{x} \in [-5,10]^3$ & $g(\mvec{x}^*) = 0$  \\ \hline
		
		12) Shekel $5$ \cite{Bingham2013}  & $\mvec{x} \in [0,10]^4$ & $ g(\mvec{x}^*) = -10.1532$
		\\ \hline
		
		13) Shekel $7$ \cite{Bingham2013}  & $\mvec{x} \in [0,10]^4$ & $g(\mvec{x}^*) = -10.4029$
		\\ \hline
		
		14) Shekel $10$ \cite{Bingham2013} & $\mvec{x} \in [0,10]^4$ & $g(\mvec{x}^*) = -10.5364$
		\\ \hline
		
		15) $^*$Shubert \cite{Bingham2013}  & $\mvec{x} \in [-10,10]^2$ & $g(\mvec{x}^*) = -186.7309$ \\ \hline
		
		16) Six-hump camel \cite{Bingham2013} & \pbox{20cm}{$x_1 \in [-3,3]$ \\ $x_2 \in [-2,2]$} & $g(\mvec{x}^*) = -1.0316$ \\ \hline
		
		17) Styblinski-Tang \cite{Bingham2013}   & $\mvec{x} \in [-5,5]^4$ & $g(\mvec{x}^*) = -156.664$  \\ \hline
		
		18) Trid  \cite{Bingham2013} & $\mvec{x} \in [-25,25]^5$ & $g(\mvec{x}^*) = -30$ \\ \hline
		
		19) Zettl \cite{Ernesto2005}   & $\mvec{x} \in [-5,5]^2$ & $g(\mvec{x}^*) = -0.00379$  \\ \hline
	\end{tabular}
\end{table}


\newpage

\newpage
\section{Additional experiments} \label{app: more_experiments}
We conducted three more experiments to test REGO's robustness to changes in the parameters. In all three experiments, the same budget and termination criteria as in the main experiment are used.
\begin{enumerate}[label={(\Alph*)}]
	\item In this experiment, we assume that no good estimate for $\mu$ is known and that $\mu$ can be as large as $\sqrt{D}$ (for example, when $\mathcal{X} = [-1, 1]^D$ constraint is imposed). We test REGO with the following parameters: $( d_e, 8.0\times \sqrt{D} ), (d_e+1, 2.2 \times \sqrt{D}), (d_e+2, 1.3 \times \sqrt{D})$ and $(d_e+3, 1.0 \times \sqrt{D})$. Results are presented in \Cref{fig: REGO_size_D} in Appendix \ref{app: more_experiments}. 
	\item We fix $\delta$ to be 7.5$\sqrt{d_e}$ and vary $d$. The following parameters are used: $( d_e, 7.5\times \sqrt{d_e} ), (d_e+1, 7.5 \times \sqrt{d_e}), (d_e+2, 7.5 \times \sqrt{d_e})$ and $(d_e+3, 7.5 \times \sqrt{d_e})$. Results are presented in \Cref{fig: REGO_fixed_delta} in Appendix \ref{app: more_experiments}.
	\item We fix $d = d_e+1$ and vary $\delta$ $(= 5\sqrt{d_e},7.5\sqrt{d_e},10\sqrt{d_e})$. The following parameters are used: $( d_e + 1, 5\times \sqrt{d_e} ), (d_e+1, 7.5 \times \sqrt{d_e})$ and  $(d_e+1, 10 \times \sqrt{d_e})$. In the figures we also include curves for $\delta_{opt} = 2.2\sqrt{d_e}$ taken from the main experiment. Results are presented in \Cref{fig: REGO_d=1} in Appendix \ref{app: more_experiments}.
\end{enumerate}
\subsubsection*{Conclusions}
\begin{enumerate}[label={(\Alph*)}]
	\item We test robustness of REGO assuming that $\mu$ is equal to $\sqrt{D}$ (which makes $\delta$ to be relatively large and dependent on $D$). Despite this dependence, the frequency of convergence for BARON and KNITRO is high showing mild dependence on $D$.
	
	\item The purpose of this experiment is to see how different values of $d$ affect the performance of REGO while $\delta$ is kept constant. For larger $d$, we expect \eqref{eq: REGO} to be successful with higher chance. Nonetheless, the results show that sometimes, for larger $d$, REGO's performance may be compromised; this is for example true for BARON's convergence$_{opt}$. Since $\delta$ is set to a relatively large value, \eqref{eq: REGO} is successful with high probability even for smallest $d$. This suggests that as long as $d$ and $\delta$ produce relatively high chance of success of \eqref{eq: REGO}, one should stop increasing their values lest convergence to the global minimum require larger computational resources. 
	
	\item In this experiment, we apply REGO with different values of $\delta$ while keeping $d$ constant. The results display no significant differences between the performances with different parameters. Even the results with the optimal $\delta$ (used in the main experiment) do not differ considerably from the one with the largest $\delta$ except for BARON where the former wins in terms of convergence$_{opt}$ and CPU time. The results of this experiment together with the results in (B) indicate that it is better to increase $\delta$ and keep $d$ constant if one wants to increase success of \eqref{eq: REGO} with minimal increase in computational cost. 
\end{enumerate}
\begin{figure}[H]
	\centering
	a) DIRECT \par\medskip
	\includegraphics[scale=0.6]{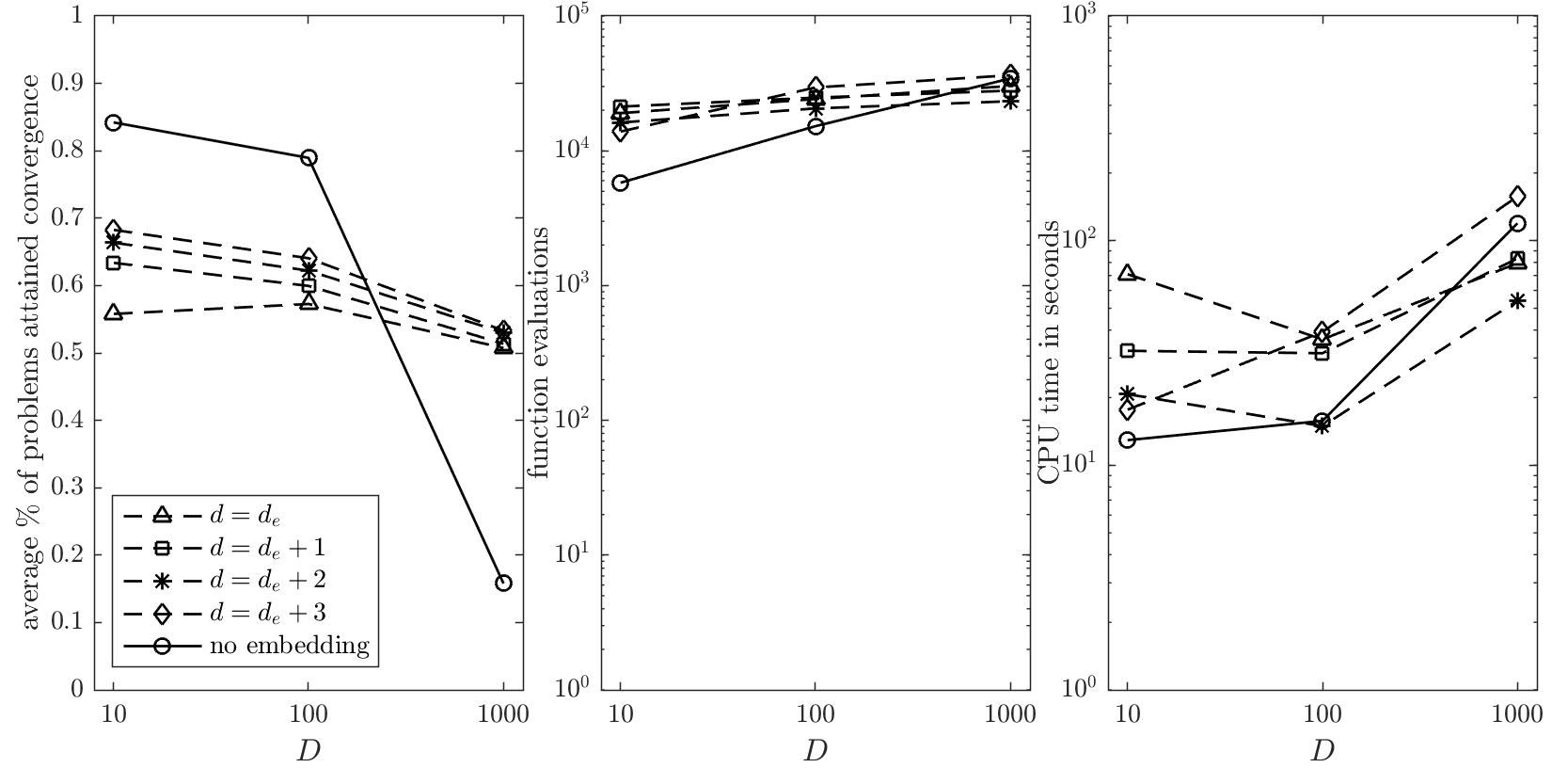} \\
	b) BARON \par\medskip
	\includegraphics[scale=0.6]{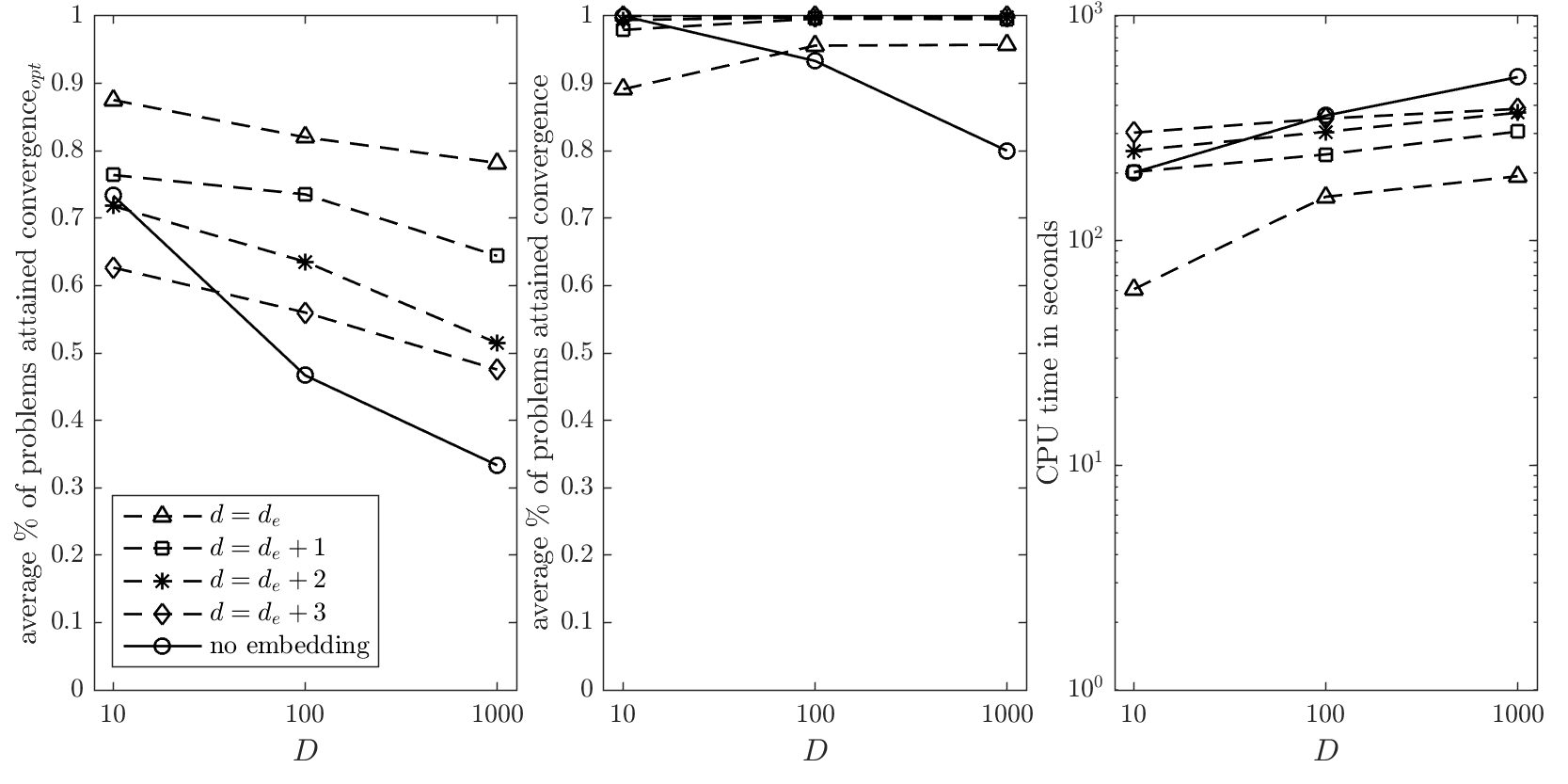} \\
	c) KNITRO \par\medskip
	\includegraphics[scale=0.6]{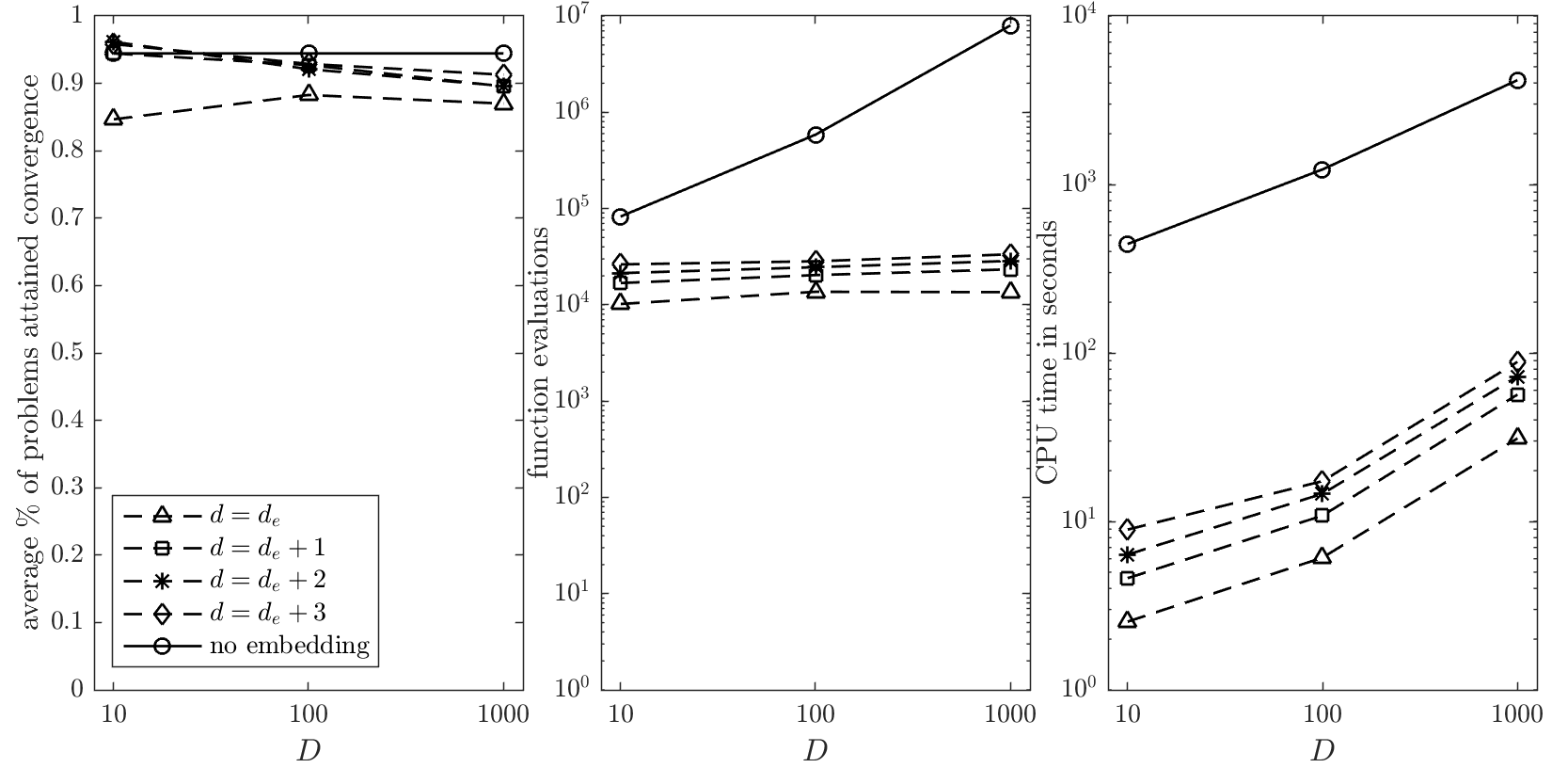}
	\caption{Experiment A: REGO with DIRECT, BARON and KNITRO with $(d,\delta) =  ( d_e, 8.0\times \sqrt{D} ), (d_e+1, 2.2 \times \sqrt{D}), (d_e+2, 1.3 \times \sqrt{D})$ and $(d_e+3, 1.0 \times \sqrt{D})$.}
	\label{fig: REGO_size_D}
\end{figure}
\begin{figure}[H]
	\centering
	a) DIRECT \par\medskip
	\includegraphics[scale=0.6]{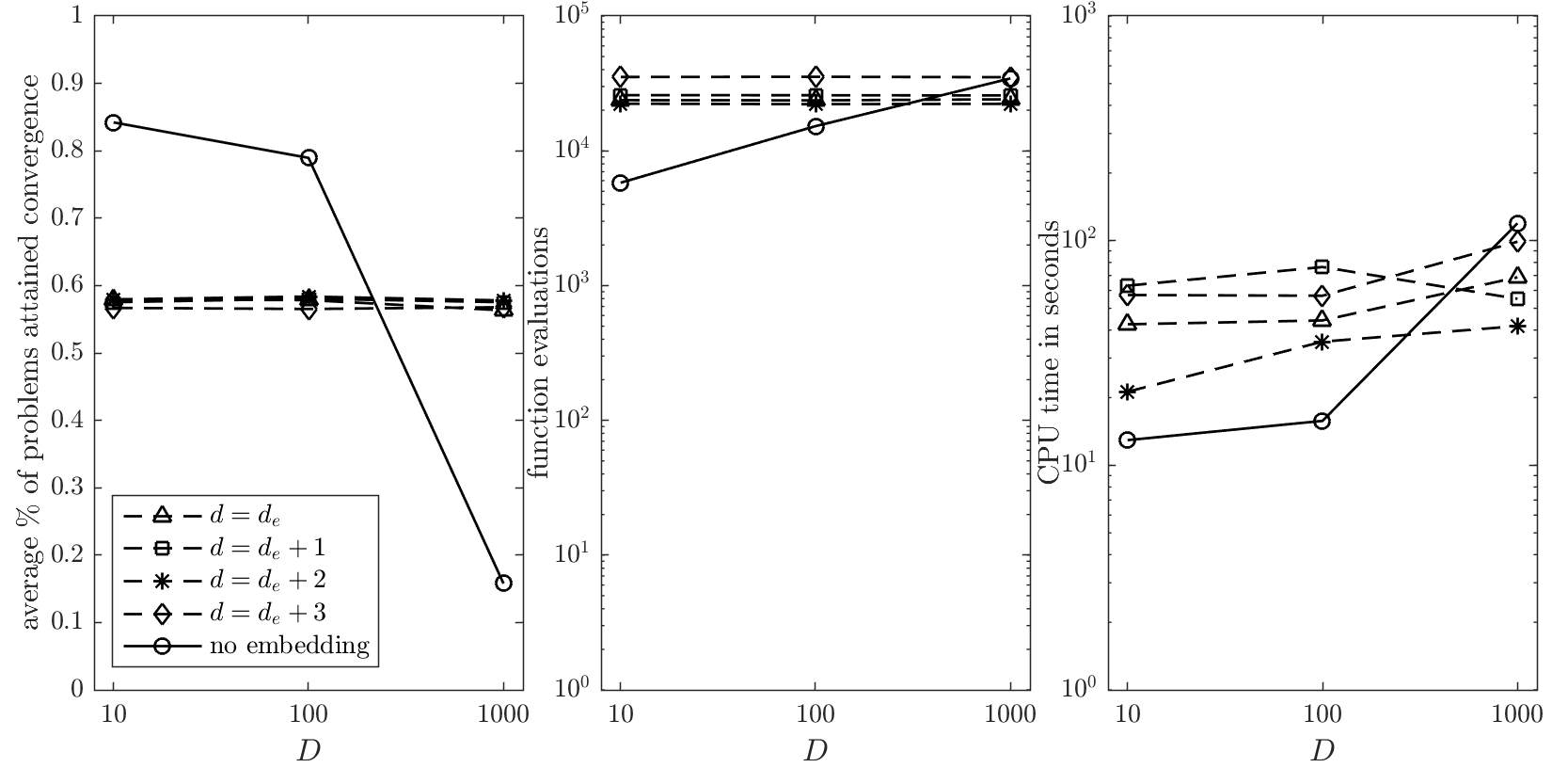} \\
	b) BARON \par\medskip
	\includegraphics[scale=0.6]{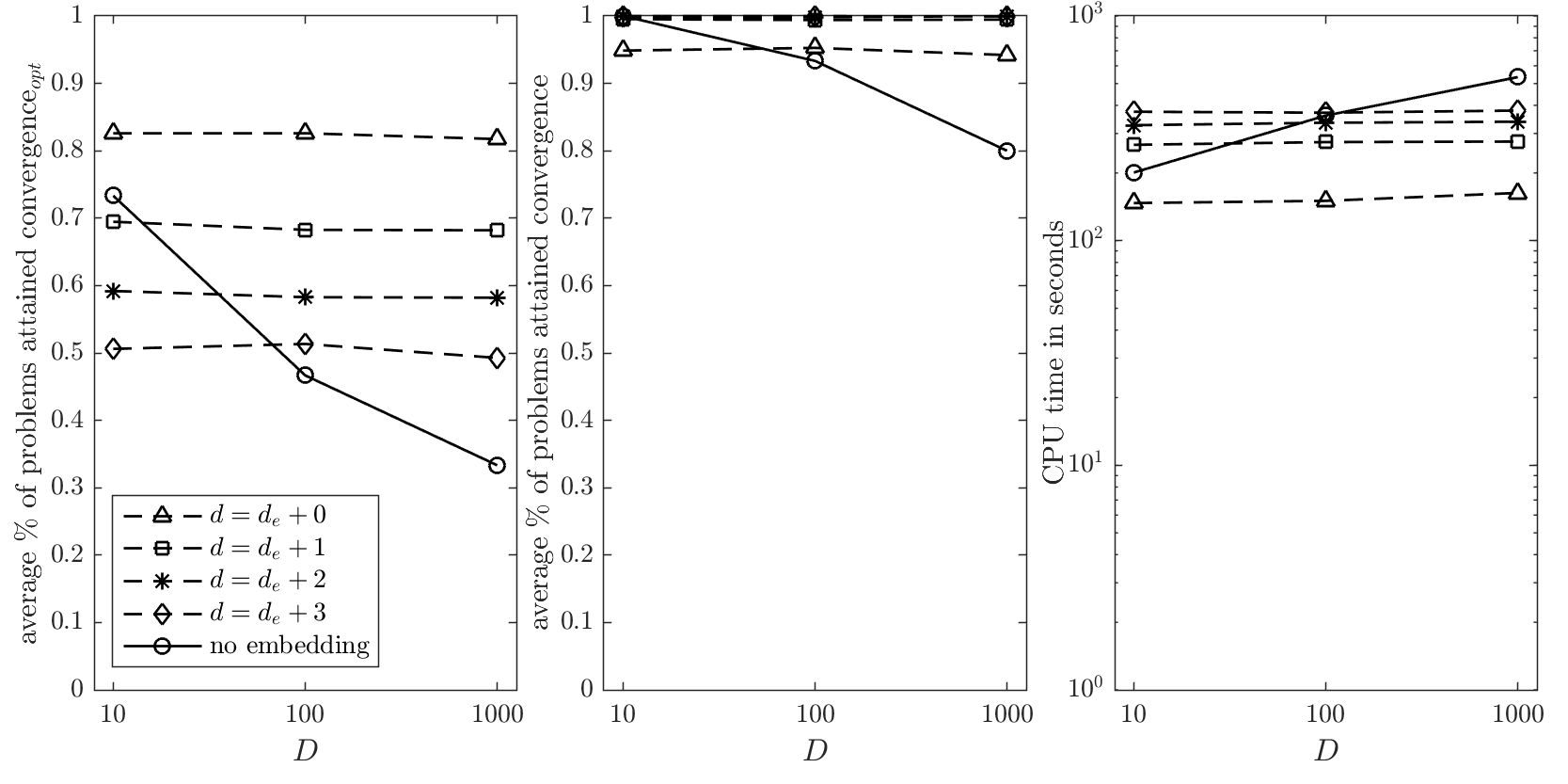} \\
	c) KNITRO \par\medskip
	\includegraphics[scale=0.6]{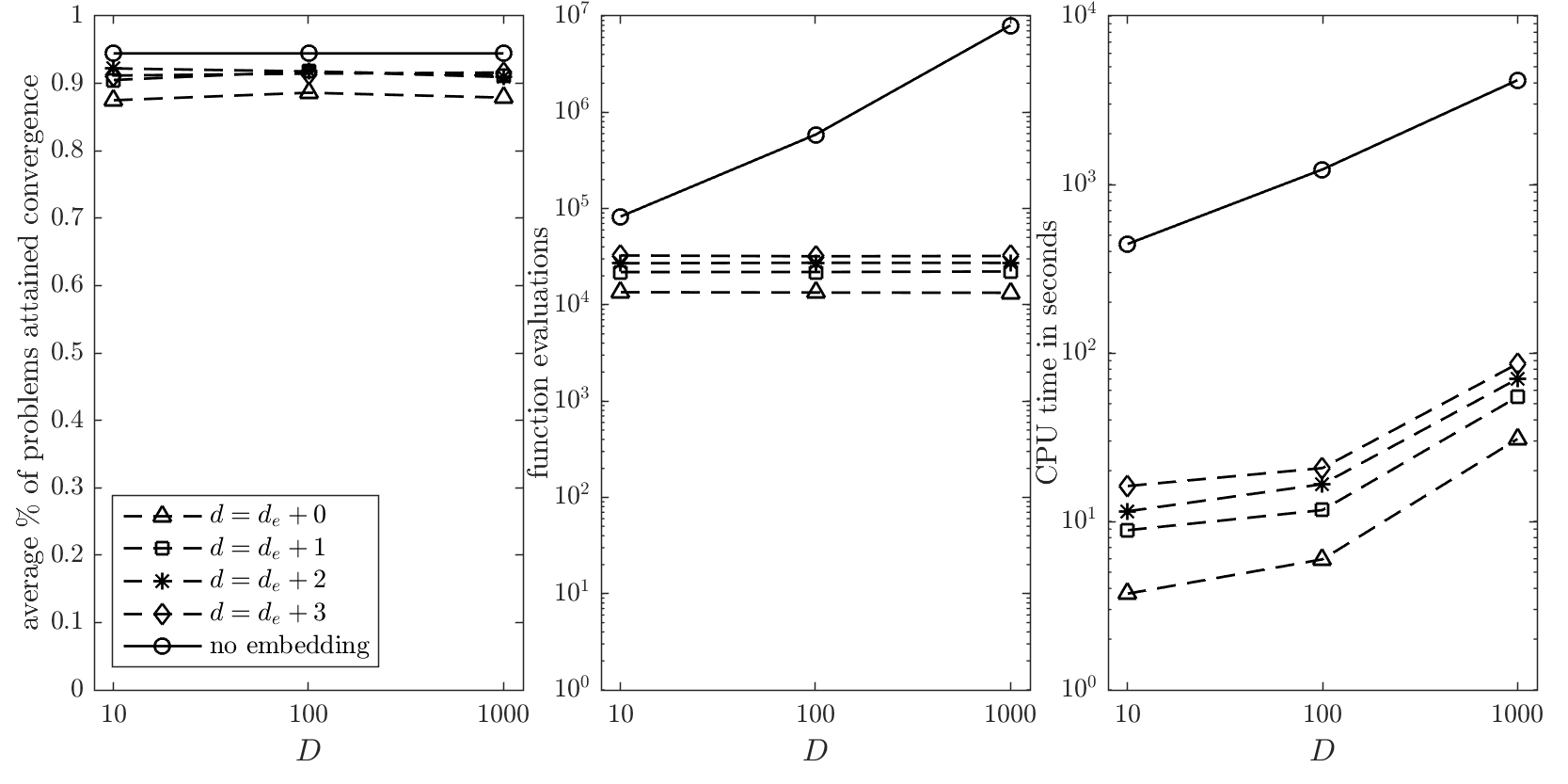}
	\caption{Experiment B: REGO with DIRECT, BARON and KNITRO with $\delta = 7.5\sqrt{d_e}$ fixed and $d = d_e$, $d_e + 1$, $d_e + 2$ and $d_e+3$.}
	\label{fig: REGO_fixed_delta}
\end{figure}
\begin{figure}[H]
	\centering
	a) DIRECT \par\medskip
	\includegraphics[scale=0.6]{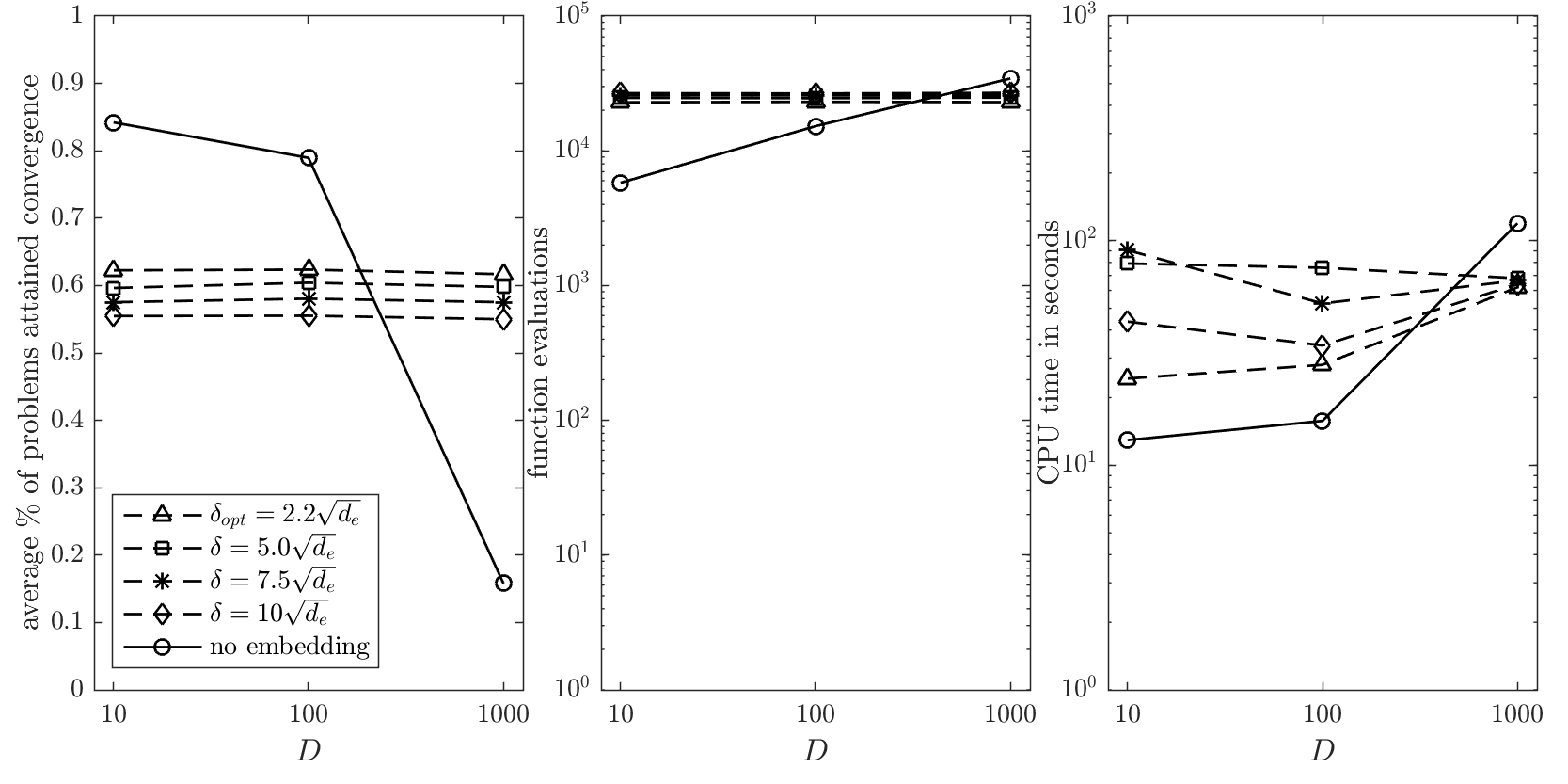} \\
	b) BARON \par\medskip
	\includegraphics[scale=0.6]{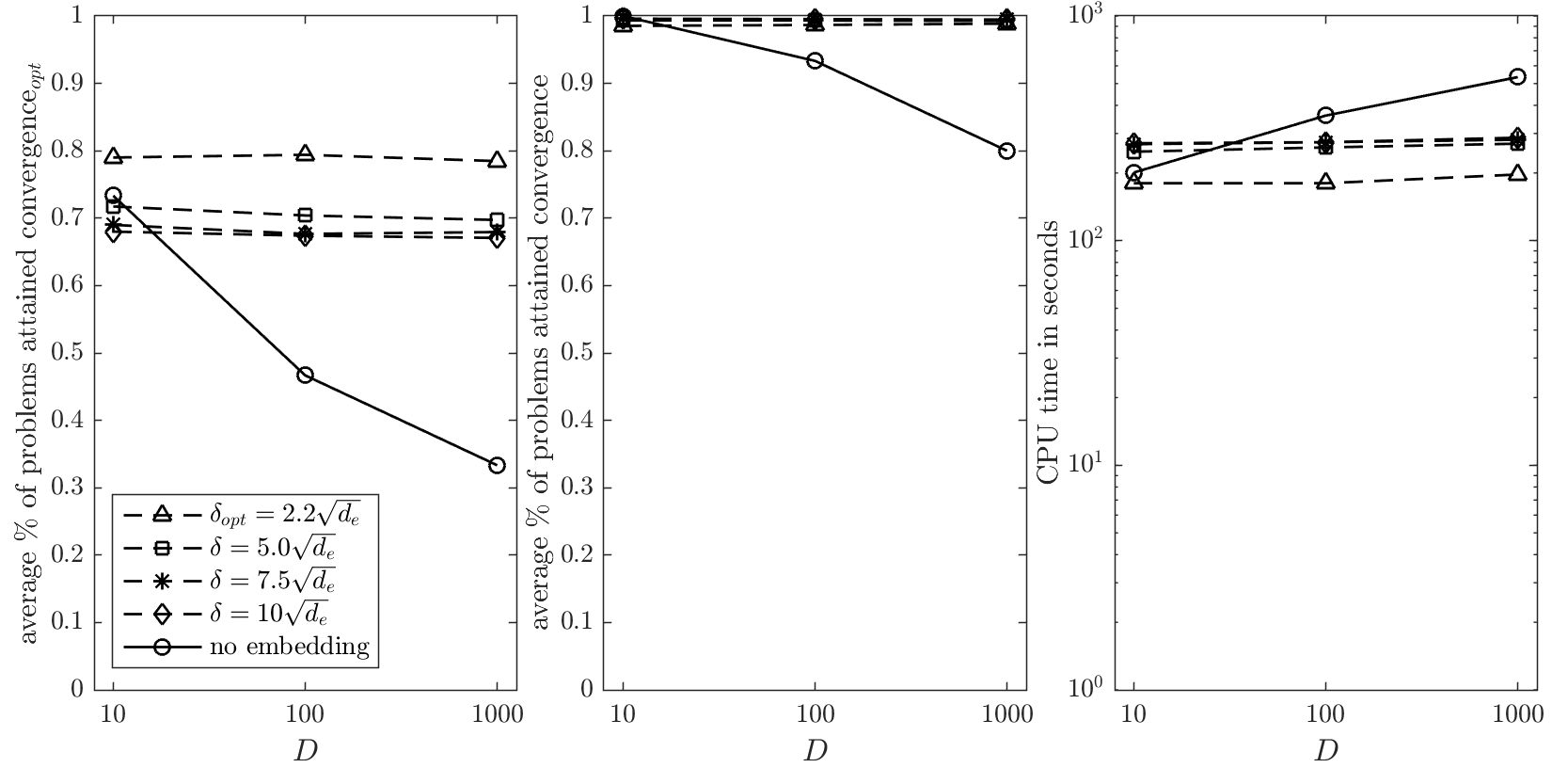} \\
	c) KNITRO \par\medskip
	\includegraphics[scale=0.6]{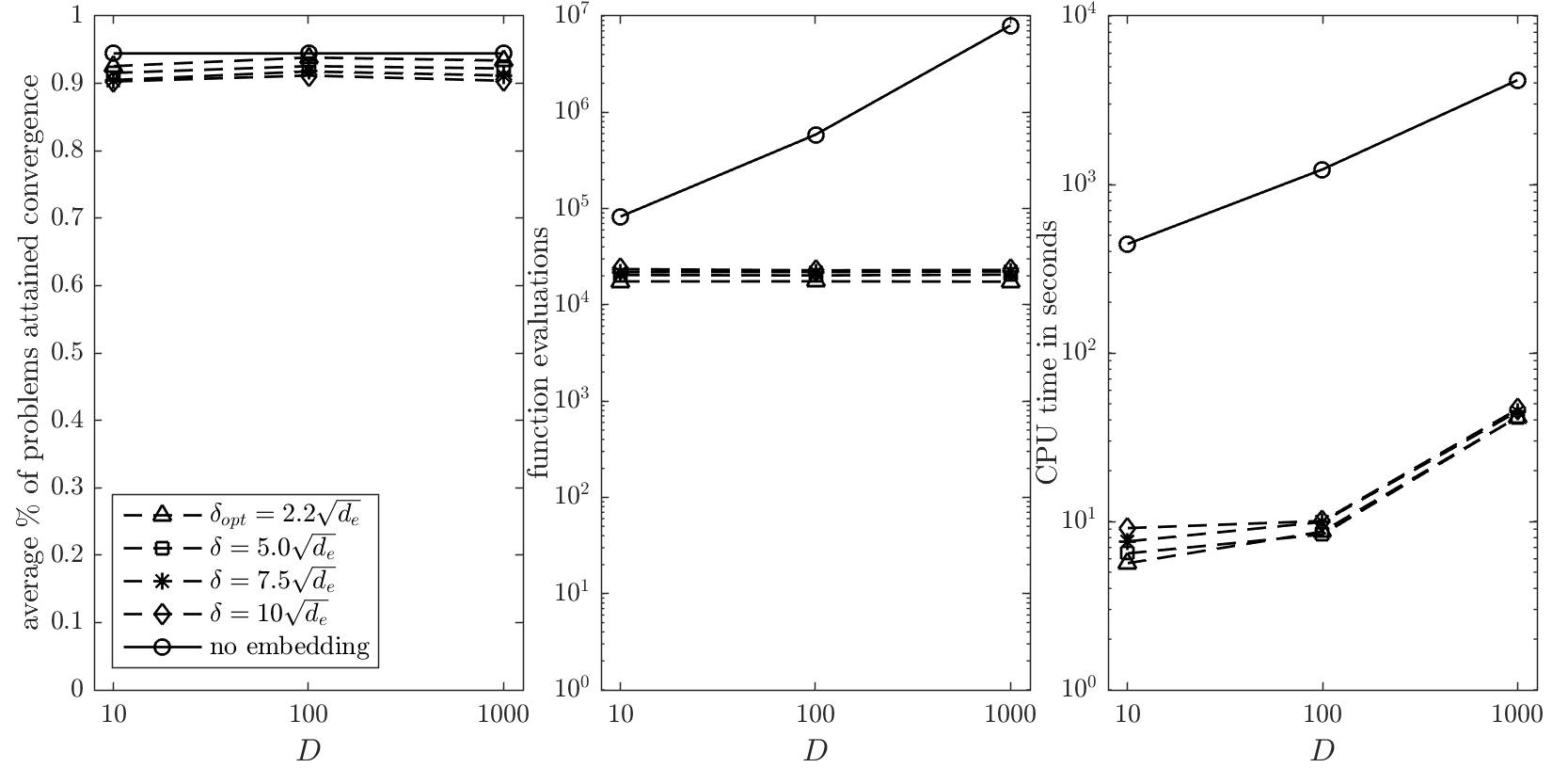}
	\caption{Experiment C: REGO with DIRECT, BARON and KNITRO with $d = d_e + 1$ fixed and $\delta = 5\sqrt{d_e}, 7.5\sqrt{d_e},10\sqrt{d_e}$ and $2.2\sqrt{d_e}$  $(\delta_{opt})$.}
	\label{fig: REGO_d=1}
\end{figure}

\end{document}